\documentclass[smallcondensed]{svjour3}     

\usepackage{color}
\usepackage{graphicx}
\usepackage{amsmath}
\usepackage{amssymb}
\usepackage[numbers]{natbib}
\usepackage{bbm}
\usepackage{bm}

\smartqed 

\journalname{} 
\begin{document}

\title{First passage times of two-dimensional correlated processes: analytical results for the Wiener process and a numerical method for diffusion processes.} 
\titlerunning{First passage times for bivariate diffusion processes}

\author{Laura Sacerdote \and Massimiliano Tamborrino\and Cristina Zucca}
\authorrunning{Sacerdote, Tamborrino, Zucca} 

\institute{L. Sacerdote \and C. Zucca \at Department of Mathematics \lq\lq G. Peano\rq\rq, University of Turin, V. Carlo Alberto 10, Turin, Italy\\
\email{laura.sacerdote@unito.it,cristina.zucca@unito.it}
\and M.Tamborrino \at Institute for Stochastics, Johannes Kepler University, 
Altenbergerstra{\ss}e 69, 4040 Linz, Austria \\
Tel.: +43 732 2468 4168\\
Fax: +43 732 2468 4162 \\
\email{massimiliano.tamborrino@jku.at}}
\date{}
\maketitle

\begin{abstract}
Given a two-dimensional correlated diffusion process, we determine the joint density of the first passage times of the process to some constant boundaries. This quantity depends on the joint density of the first passage time of the first crossing component and of the position of the second crossing component before its crossing time. First we show that these densities are solutions of a system of Volterra-Fredholm first kind integral equations. Then we propose a numerical algorithm to solve it and we describe how to use the algorithm  to approximate the joint density of the first passage times. The convergence of the method is theoretically proved for bivariate diffusion processes.
We derive explicit expressions for these and other quantities of interest in the case of a bivariate Wiener process, correcting previous misprints appearing in the literature. Finally we illustrate the application of the method through a set of examples. \medskip

{\bf This paper replaces the old version \lq\lq First passage times of two-dimensional correlated diffusion processes: analytical and numerical methods \rq\rq. Please site this article as L. Sacerdote, M. Tamborrino, and C. Zucca. First passage times of two-dimensional correlated processes: analytical results for the Wiener process and a numerical method for diffusion processes, Journal of Computational and Applied Mathematics,  296, 275-292, 2016.}\medskip

\end{abstract}
\medskip

\medskip

\section{Introduction and motivation}\label{section1}
The first passage time (FPT) problem of  univariate stochastic processes through boundaries is relevant in different fields, e.g. economics \citep{econ1}, engineering \citep{eng}, finance \citep{applbook,finance}, neuroscience \citep{ReviewSac,Tpsycho}, physics \citep{bookFPT}, psychology \citep{Navarro} and reliability theory \citep{rel,PieperDom}. For  one-dimensional  processes, the FPT problem has been widely analytically investigated both for constant and time dependent boundaries \citep{Alili, Martin,ReviewSac}, yielding explicit expressions for the FPT density of the Wiener process \cite{coxMiller}, of a special case of the Ornstein Uhlenbeck (OU) process \cite{Ricciardi}, of the Cox-Ingersoll-Ross (also known as Feller or square-root) process  \cite{CapRic,Sacerdote1990} and of some processes which can be obtained through suitable measure or space-time transformations of the previous processes \cite{Alili,CapRic,RicciardiW}.  For most of the processes arising from applications, closed form expressions are not available but it was proved that the FPT distribution function is solution of  integral equations. This has determined the development of ad hoc numerical methods for the solution of Volterra integral equations of the first and second types arising from  both the direct and the inverse FPT problem \cite{Buonocore,Gob,Milst,RicciardiNip,Telve,Zucca}.

Results for the FPT problem of bivariate processes are still scarce and fragmentary. Analytic results are available for bivariate FPTs through specific surfaces \citep{Dicrescenzo,Lachal}, for the FPTs of a Wiener and of an integrated process \citep{Dynkin,Gard,Lefebvre1,Lefebvre2} and for the FPTs of two correlated Wiener processes with zero \cite{Buckholtz, Iyengar,Shao} or positive drift \cite{Domine} in presence of absorbing boundaries. 

The main goal of this paper is  to investigate the bivariate joint distribution of the hitting times of a bivariate diffusion process. A new difficulty arises with respect to the univariate case: the dynamics of the process after the first crossing depend on the type of considered boundaries. Indeed the first component attaining its boundary can stop its evolution, be absorbed there or pursue its evolution depending on whether the boundaries are \textit{killing}, \textit{absorbing} or \textit{crossing}, respectively. In all cases, the slowest component evolves till its passage time. The different boundary conditions are defined in Section \ref{Section2} together with some further mathematical background. The different scenarios are studied in Section \ref{section3}, where we also derive the joint FPT densities in the three cases.
Conscious of the important role of Volterra Integral equations in the univariate FPT problem, here we extend the approach used in the one dimensional case \cite{Buonocore} or in the FPT problem of a component of a Gauss Markov process \cite{Benedetto}. These quantities depend on the joint densities of the second crossing component before its FPT and the FPT of the first crossing component, which we show to be the solutions of a system of Volterra-Fredholm first kind integral equations \cite{Fre-Volt}. In Section \ref{section4} we propose a numerical method to solve the system and we describe how to obtain the joint FPT density using our algorithm. Since the dynamics of the process before the first crossing time are the same for all types of boundaries, the proposed method can always be used. In Section \ref{section5} we prove the convergence of the algorithm and we study its order of convergence. A useful feature of the proposed algorithm is that it allows to avoid the prohibitive computational effort required for simulating  the joint density of the FPTs \cite{Zhou}. Indeed it allows to switch from a Monte Carlo simulation method \cite{Metzler} to a deterministic numerical method.

To numerically illustrate the convergence of the method, we consider two correlated Wiener processes and compare the theoretical and the numerical results. The desired joint density of the second crossing component before its FPT and the FPT of the first crossing component can be obtained starting from the joint density of the process constrained to be below the boundaries, which is available in \cite{Domine,Iyengar,Metzler}. The formulas for the driftless case presented in \cite{Iyengar} contain misprints, which have been independently corrected in \cite{Domine} and \cite{Metzler}. In \cite{Iyengar} the case with drift is also considered, but unfortunately some expressions present further misprints. Since we have not been able to locate correct results elsewhere in the literature, in Section \ref{section6}  we correct these formulas and determine other quantities of interest. In particular we calculate the joint density of the position of the process constrained to be below the boundaries, of the FPTs both with and without drift and of the second crossing component before its FPT and the FPT of the first crossing component. A comparison of this last density with its numerical approximation obtained using the algorithm is presented in Section \ref{section7}. There we also illustrate the application of our method to approximate the joint FPT density of a bivariate OU process with correlated components. This is particularly relevant in neuroscience, where FPTs are used to describe neural action potentials (spikes) and multivariate OU processes can be used to model neural networks, as recently discussed in \cite{TSJ}.

\section{Mathematical background}\label{Section2}
Consider a two-dimensional time homogeneous diffusion
process $\mathbf{X}=\left\{ (X_1,X_2)'(t); t>t_0\right\}$, solution of the stochastic differential equation
\begin{equation}\label{diff}
d\mathbf{X}(t)=\bm \mu(\mathbf{X}(t))dt+\bm \Sigma\left(\bm X(t)\right) d\mathbf{W}(t), \qquad \bm{X}(t_0)=\bm{x}_0=\left(
x_{01,}x_{02}\right)', \qquad t>t_0,
\end{equation}
where $'$ indicates vector transpose. Here $\bm W(t)$ is a two-dimensional standard Wiener process, the $\mathbb{R}^2$-valued function $\bm\mu$ and the $\mathbb{R}^2\times \mathbb{R}^2$ matrix-valued  function  $\bm\Sigma$ are assumed to be defined and measurable on $\mathbb{R}^2$ and all the conditions on existence and uniqueness of the solution are satisfied \cite{arnold}.

Define the random variable
\begin{equation*}
T_{i}=\inf \{t>t_0:X_{i}\left( t\right) >B_{i}\} \qquad i=1,2,
\end{equation*}
i.e. the FPT of $X_i$ through the constant boundary $B_i>x_{0i}$. We denote by $T=\min(T_1,T_2)$ the random variable corresponding to the first exit time of $\mathbf{X}$ from the strip $(-\infty,B_1)\times (-\infty,B_2)$. Our goal is to determine the joint probability density function (pdf) of $(T_1,T_2)$ for a process $\mathbf{X}$ originated in $\mathbf{y}=(y_1,y_2)$ at time $s$, defined by
\[
f_{(T_1,T_2)}(t_1,t_2|\bm y, s):=\frac{\partial^2}{\partial t_1 \partial t_2} \mathbb{P}(T_1<t_1, T_2<t_2|\bm X(s)=\bm y).
\]
Throughout the paper we consider the following densities for $i,j=1,2, i\neq j$ and $s<t$:
\begin{itemize}
\item[$\bullet$]  joint pdf of the components of the process $\mathbf{X}$ up to time $T$, defined by
\[
f^a_{\mathbf{X}}(\mathbf{x},t|\mathbf{y},s):=\frac{\partial^2}{\partial x_1\partial x_2} \mathbb{P}\left(\mathbf{X}(t)<\mathbf{x},T>t \vert \mathbf{X}(s)=\mathbf{y}\right);
\]
\item[$\bullet$] conditional pdf of $X_i$ given $X_j$ up to  time $T$, defined by
\begin{eqnarray*}
&&
f^a_{X_i\vert X_j}(x_i,t\vert x_j,t;\mathbf{y},s):=\frac{\partial}{\partial x_i} \mathbb{P}(X_i(t)< x_i,T_i>t\vert X_j(t)=x_j, T_j>t,\mathbf{X}(s)=\mathbf{y});\\
 \end{eqnarray*}
 \item[$\bullet$] conditional pdf of $X_i$ up to time $T_i$ given $T_j$, defined by
\begin{eqnarray*}
f^a_{X_{i}\left\vert T_{j}\right. }\left( x_i \vert t;\mathbf{y},s
\right):=\frac{\partial}{\partial x_i}\mathbb{P}\left( X_i(t)< x_i,T_i>t \vert T_{j}=t, \mathbf{X}(s) =\mathbf{y} \right);
\end{eqnarray*}
\item[$\bullet$] joint pdf of $X_i$ up to time $T_i$ and of $T_j$, defined by
\begin{eqnarray*}
f^a_{(X_{i}, T_{j})}\left( x_i,t\vert \bm y,s\right):=\frac{\partial^2}{\partial x_i \partial t} \mathbb{P}\left( X_i(T_{j}) <  x_i, T_i>t, T_{j} < t \vert \bm X(s) =\bm y , T>s\right);
\end{eqnarray*}
\item[$\bullet$] joint pdf of $X_i$ and $T_j$, defined by
\[
f_{(X_i, T_{j})}\left(x_i, t| \bm y,s\right):=\frac{\partial^2}{\partial x_i\partial t}\mathbb{P}\left( X_{i}\left( T_{j}\right) < x_i, T_{j} <t \vert \bm X \left( s\right) =\bm y \right);
\]
\item[$\bullet$] marginal pdf of $T_i$, defined by 
\[
f_{T_i}(t\vert \mathbf{y},s):=\frac{\partial}{\partial t}\mathbb{P}(T_i< t\vert \mathbf{X}(s)=\mathbf{y}).
\]
\end{itemize}
Finally we denote the survival cumulative distribution function of $\mathbf{X}$ by $\bar{F}_{\mathbf{X}}(\mathbf{x}, t|\mathbf{y}, s)=\mathbb{P}(\mathbf{X}(t)>\mathbf{x}|\mathbf{X}(s)=\mathbf{y})$ and its transition pdf by $f_\mathbf{X}(\mathbf{x}, t|\mathbf{y}, s)$ for $s<t, \mathbf{x}, \mathbf{y}\in\mathbb{R}^2$. 
To simplify the notation, we omit to write the starting position when $\mathbf{y}=\mathbf{x}_0$, and the starting time when $s=t_0$. All the above densities are assumed to exist, to be well defined and either known in closed form or numerically evaluable.

\subsection{Behavior of the process in presence of different types of boundaries}
The behavior of the subthreshold process $\mathbf{X}$ up to time $T$ does not depend on the type of boundaries, while the dynamics of the process after time $T$ do.  Denote by $X_f$ and $X_s$ the first (and thus the fastest) and second (and thus the slowest) components to hit their boundaries $B_f$ and $B_s$ at time $T=T_f$ and $T_s$, respectively.  Throughout the paper we consider the next  three possible scenarios, as illustrated in Fig. \ref{fig}:
\begin{enumerate}
\item \emph{Killing boundaries}. At time $T_f$ the fastest component $X_f$ is killed while the slowest component $X_s$ pursues its evolution till time $T_s$. Thus after $T_f$ the process becomes univariate and $X_s$ does not depend on $X_f$ anymore.
\item \emph{Absorbing boundaries}. At time $T_f$ the fastest component $X_f$ is absorbed in $B_f$  while the slowest component $X_s$ pursues its evolution till time $T_s$. After $T$, the process is still bivariate but the component $X_f$ is constant, i.e. $X_f(u)=B_f$ for $u \in [T_f,T_s]$.
\item \emph{Crossing boundaries}. Both components continue to evolve according to \eqref{diff} also after time $T$.
\end{enumerate}

\begin{figure}[t!]
\centering
\includegraphics[width=.8\textwidth]{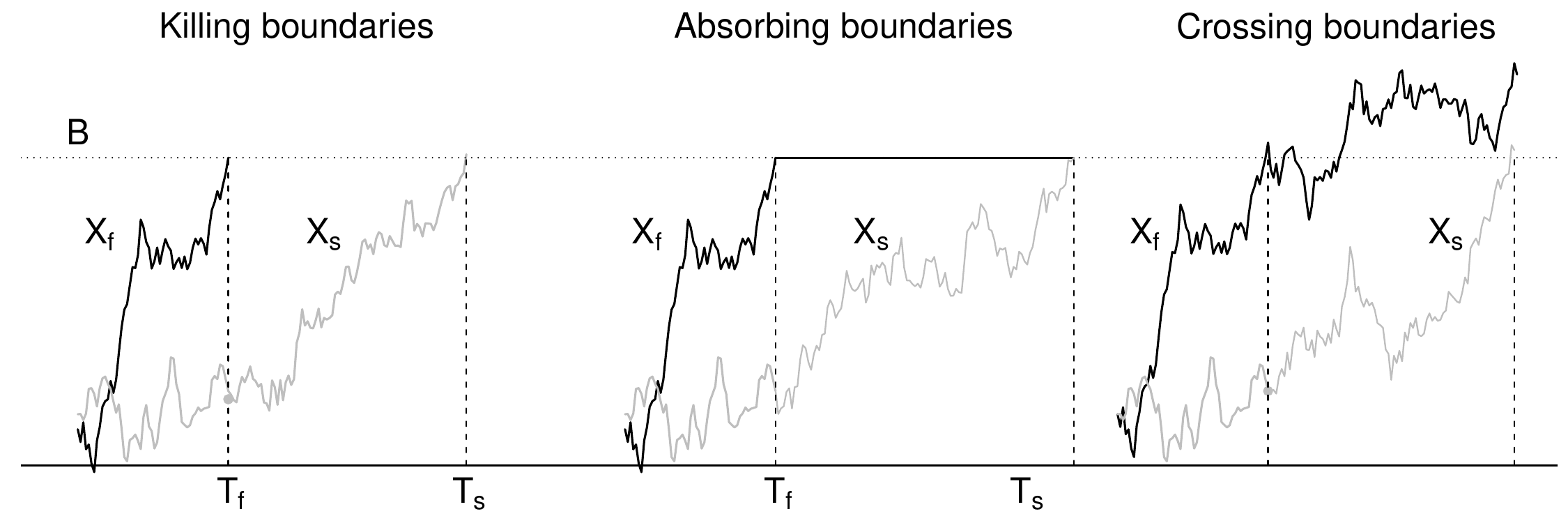}
\caption{Schematic representation of the dynamic of the bivariate diffusion process $\mathbf{X}$ in presence of killing (left figure), absorbing (central figure) or crossing constant boundaries (right figure) $B_1=B_2=B$. $X_f$ denotes the first (and thus the fastest) component hitting its boundary at time $T_f$, $X_s$ the second (and thus the slowest) to hit the boundary at time $T_s$. In presence of killing boundaries (left figure), only the slowest component $X_s$ evolves after time $T_f$ and the process becomes univariate. In presence of absorbing boundary (central figure), the fastest component $X_f$ is absorbed on its boundary $B_f$, while $X_s$ pursues its evolution until its FPT. Hence the process is still bivariate with a constant component $X_f(u)=B_f$ for $u\in[T_f,T_s]$. In presence of crossing boundaries (right figure), both components evolve until the time when both have reached their levels. }
\label{fig}
\end{figure}
\section{Joint distribution of $\left( T_{1},T_{2}\right) $}\label{section3}
Since the dynamics of $\mathbf{X}$ are the same up to $T_f$ but depend on the considered boundaries in $(T_f,T_s)$, the joint FPT density $f_{(T_1,T_2)}$ has three different  expressions as shown by the following 
\begin{theorem}\label{theo} 
When $T_1\neq T_2$, the joint density of $(T_1,T_2)$ is 
{\footnotesize{\begin{equation}\label{ProbTTa}
f_{(T_1, T_2)}(t_1,t_2)= \int_{-\infty }^{B_2} f_{T_2}(t_2\vert x_2,t_1)f^a_{(X_{2}, T_{1})}(x_{2},t_{1})dx_2\mathbbm{1}_{\{t_1<t_2\}}+ \int_{-\infty }^{B_1} f_{T_1}(t_1\vert x_1,t_2)f^a_{(X_{1}, T_{2})}(x_{1},t_{2})dx_1 \mathbbm{1}_{\{t_2<t_1\}},
\end{equation}}}
if the boundaries are killing,
{\footnotesize{\begin{equation}\label{ProbTT}
f_{(T_1, T_2)}(t_1,t_2)= \int_{-\infty }^{B_2} f_{T_2}(t_2\vert (B_1,x_2),t_1)f^a_{(X_{2}, T_{1})}(x_{2},t_{1})dx_2 \mathbbm{1}_{\{t_1<t_2\}}+ \int_{-\infty }^{B_1} f_{T_1}(t_1\vert (x_1,B_2),t_2)f^a_{(X_{1}, T_{2})}(x_{1},t_{2})dx_1 \mathbbm{1}_{\{t_2<t_1\}}
\end{equation}}}
if the boundaries are absorbing and
{\footnotesize{\begin{eqnarray}\label{ProbTT3}
f_{(T_1, T_2)}(t_1,t_2) &= &\int_{-\infty }^{B_2}\left( \int_{-\infty}^{\infty} f_{( X_1,T_2)}((x_1,B_2),t_2|(B_1,x_2), t_1)dx_1\right) f^a_{(X_2,T_1)} (x_2,t_1)dx_2 \mathbbm{1}_{\{t_1<t_2\}}\nonumber\\
&+&\int_{-\infty }^{B_1} \left(\int_{-\infty}^{\infty} f_{( X_2,T_1)}((B_1,x_2),t_1|(x_1,B_2), t_2)dx_2\right)f^a_{(X_1,T_2)} (x_1,t_2)dx_1 \mathbbm{1}_{\{t_2<t_1\}}
\end{eqnarray}}}
if the boundaries are crossing. Here $\mathbbm{1}_{A}$ denotes the indicator function of the set $A$. 
\end{theorem}
The proof of Theorem \ref{theo} is given in Appendix A. 

\begin{remark}
Note that \eqref{ProbTTa} is a particular case of \eqref{ProbTT} when $X_i$ does not depend on $X_j$ after time $T_j>T_i$.
\end{remark}
\begin{remark}
Since the behavior of $\mathbf{X}$ does not depend on the boundaries up to time $T$, all the pdfs of $f_{(T_1,T_2)}$ in \eqref{ProbTTa}-\eqref{ProbTT3} are functions of $f_{(X_2^a,T_1)}$  and $f^a_{(X_1, T_2)}$. These densities can be obtained as solutions of systems of Volterra-Fredholm integral equations. 
\end{remark}
\begin{theorem}\label{Volt}
In presence of crossing boundaries the densities $f^a_{(X_{1}, T_{2})}$ and $f^a_{(X_{2}, T_{1})}$ are solutions of the following system of Volterra-Fredholm first kind integral equations
\begin{subequations}\label{Volterra}
{\footnotesize{\begin{align}
\bar{F}_{\mathbf{X}}((x_1,B_2),t)&=\int_{t_0}^t \int_{-\infty}^{B_2} \bar{F}_{\mathbf{X}}((x_1,B_2),t\vert(B_1,y),\tau)f^a_{(X_{2}, T_{1})}\left( y, \tau\right)dy d\tau+\int_{t_0}^t \int_{-\infty}^{B_1}  \bar{F}_{\mathbf{X}}((x_1,B_2),t\vert(y,B_2),\tau)f^a_{(X_{1},T_{2})}\left( y, \tau\right)dy d\tau;  \label{Volterra1}\\
\bar{F}_{\mathbf{X}}((B_1, x_2),t)
&=\int_{t_0}^t \int_{-\infty}^{B_2} \bar{F}_{\mathbf{X}}((B_1, x_2),t\vert(B_1,y),\tau)f^a_{(X_{2}, T_{1})}\left( y, \tau\right)dy d\tau +\int_{t_0}^t \int_{-\infty}^{B_1} \bar{F}_{\mathbf{X}}((B_1, x_2),t\vert(y,B_2),\tau)f^a_{(X_{1}, T_{2}) }\left( y,\tau\right)dy d\tau,  \label{Volterra2} 
\end{align}}}
\end{subequations}
where $x_1>B_1$ and $x_2>B_2$. Moreover, it holds 
 \begin{equation}\label{Volterradensity}
f_{\mathbf{X}}(\mathbf{x},t)=\int^{t}_{t_0} \int_{-\infty}^{B_2}f_{\mathbf{X}}\left(\mathbf{x},t\vert(B_1,y),\tau\right)
f^a_{(X_{2},T_{1})}\left( y,\tau\right)
dy d\tau +\int^{t}_{t_0} \int_{-\infty}^{B_1}f_{\mathbf{X}}\left(\mathbf{x},t\vert(y,B_2),\tau\right) f^a_{(X_{1}, T_{2} )}\left( y,\tau\right)dy d\tau,
\end{equation}
which yields another system of Volterra-Fredholm first kind integral equations.
\end{theorem}
The proof of Theorem \ref{Volt} is given in Appendix B. 
\begin{remark}
Even if the system is written with respect to crossing boundaries, the solutions $f^a_{(X_i,T_j)}$ are the same for all scenarios and can be then plugged in \eqref{ProbTTa}-\eqref{ProbTT3} to obtain the joint pdf of $(T_1,T_2)$ for killing, absorbing and crossing boundaries, respectively. Different systems can be written depending on the type of boundaries, obviously yielding the same solution. For this reason, we do not write them here. Note also that the difficulties of the solutions do not change when different equations are considered. 
\end{remark}
 
\section{Numerical method}\label{section4}
When $\mathbf{X}$ is a multivariate Gaussian process, an explicit expression for $\bar F_\mathbf{X}$ is available \cite{arnold}. Analytical  and numerical expressions of $f_{\mathbf{X}}$ can be found in \cite{PlatenBL}.
 In general, an analytic solution of the system (\ref{Volterra}) is untenable for most of the diffusion processes.  For this reason,  we develop a suitable numerical method for its solution. Consider a two-dimensional time interval $[0,\Theta_1]\times[0,\Theta_2]$, with $\Theta_1,\Theta_2\in \mathbb{R}^+$.
For each component $i=1,2$, let $h_i$ and $r_i$ be the time and space discretization steps, respectively.
On $\{(-\infty,B_1)\times(-\infty,B_2)\times[0,\Theta_1]\times[0,\Theta_2]\}$ we introduce the partition $\{(y_{u_1},y_{u_2});t_{k_1},t_{k_2}\}$ where $t_{k_i}=k_ih_i$ is the time discretization and $y_{u_i}=B_i-u_ir_i$  is the space discretization for $k_i=0,\ldots,N_i, N_ih_i=\Theta_i, u_i\in \mathbb{N}$, and $i = 1,2$. To simplify the notation, we consider $h_1=h_2=h$ and $\Theta_1=\Theta_2=\Theta$, implying $N_1=N_2=N$, $k_1=k_2=k$ and thus $t_{k_1}=t_{k_2}=t_k$, for $k=0,\ldots, N$.

We use the Euler method  \cite{Li} to approximate the time integrals in \eqref{Volterra}, obtaining a system of integral equations for $\hat{f}^a_{(X_{i}, T_{j}) }(y , t)$, which denotes the approximation of $f^a_{(X_i,T_j)}(y,t)$ due to time discretization. For $x_1>B_1$ and $x_2>B_2$, we get 
\begin{subequations}\label{fcap}
{\footnotesize{\begin{align}
 \bar{F}_{\mathbf{X}}((x_1,B_2),t_k) &=h\sum_{\rho=0}^k\left[ \int_{-\infty}^{B_2} \bar{F}_{\mathbf{X}}((x_1,B_2),t_k\vert(B_1,y),t_\rho)\hat{f}^a_{(X_{2}, T_{1}) }\left( y , t_\rho\right) dy+\int_{-\infty}^{B_1} \bar{F}_{\mathbf{X}}((x_1,B_2),t_k\vert( y,B_2),t_\rho)\hat{f}^a_{(X_{1}, T_{2}) }\left(  y , t_\rho\right)dy\right];\label{Volterra1discret} \\
 \bar{F}_{\mathbf{X}}\left((B_1, x_2),t_k\right) &=h\sum_{\rho=0}^k \left[\int_{-\infty}^{B_2} \bar{F}_{\mathbf{X}}\left((B_1, x_2\right),t_k\vert(B_1,y),t_\rho)\hat{f}^a_{(X_{2},T_{1})}\left( y, t_\rho\right)dy +\int_{-\infty}^{B_1} \bar{F}_{\mathbf{X}}((B_1, x_2),t_k\vert(y,B_2),t_\rho)\hat{f}^a_{(X_{1}, T_{2}) }( y , t_\rho) dy\right].\label{Volterra2discret}
 \end{align}}}
\end{subequations}
Note that
\begin{eqnarray}\label{Flimit}
&&\bar{F}_{\mathbf{X}}((B_1, x_2),t_k\vert(B_1,y),t_k)= \mathbbm{1}_{\{y>x_2\}}; \qquad \bar{F}_{\mathbf{X}}((x_1,B_2),t_k\vert( y,B_2),t_k)=\mathbbm{1}_{\{y>x_1\}}; \nonumber\\
[-1.5ex]\\
&&\bar{F}_{\mathbf{X}}((B_1, x_2),t_k\vert(y,B_2),t_k)=0;\qquad\qquad \bar{F}_{\mathbf{X}}((x_1,B_2),t_k\vert(B_1,y),t_k)=0. \nonumber
\end{eqnarray}
Plugging \eqref{Flimit} into \eqref{fcap} and differentiating with respect to $x_j, j=1,2$, we get the system
\begin{subequations}\label{fcaplimit}
\footnotesize{\begin{align}
\frac{\partial  \bar{F}_{\mathbf{X}}((x_1,B_2),t_k)}{\partial x_1} &=-h \hat{f}^a_{(X_{1}, T_{2}) }\left( x_1 , t_k\right)+h\sum_{\rho=0}^{k-1} \int_{-\infty}^{B_2} \frac{\partial \bar{F}_{\mathbf{X}}((x_1,B_2),t_k\vert(B_1,y),t_\rho)}{\partial x_1}\hat{f}^a_{(X_{2}, T_{1}) }\left( y , t_\rho\right) dy\nonumber\\
&+h\sum_{\rho=0}^{k-1} \int_{-\infty}^{B_1} \frac{\partial \bar{F}_{\mathbf{X}}((x_1,B_2),t_k\vert( y,B_2),t_\rho)}{\partial x_1}\hat{f}^a_{(X_{1}, T_{2}) }\left(  y , t_\rho\right)dy; \label{Volterra1discret2} \\
\frac{\partial \bar{F}_{\mathbf{X}}\left((B_1, x_2),t_k\right)}{\partial x_2} &=-h \hat{f}^a_{(X_{2}, T_{1}) }\left(  x_2 , t_k\right)+ h\sum_{\rho=0}^{k-1} \int_{-\infty}^{B_2} \frac{\partial \bar{F}_{\mathbf{X}}\left((B_1, x_2\right),t_k\vert(B_1,y),t_\rho)}{\partial x_2}\hat{f}^a_{(X_{2},T_{1})}\left( y, t_\rho\right)dy \nonumber\\
 &+h\sum_{\rho=0}^{k-1} \int_{-\infty}^{B_1} \frac{\partial \bar{F}_{\mathbf{X}}((B_1, x_2),t_k\vert(y,B_2),t_\rho)}{\partial x_2}\hat{f}^a_{(X_{1}, T_{2}) }( y , t_\rho) dy \label{Volterra2discret2}.
 \end{align}}
\end{subequations}
For multivariate Gaussian processes, the derivatives $\partial \bar F_{\mathbf{X}}/\partial x_i$ are known while analytical or numerical computations can be performed for other processes. If we discretize  the spatial integral and we truncate the corresponding series with a finite sum, we obtain
\begin{subequations}\label{ftilda1}
\footnotesize{
\begin{align}
\frac{\partial\bar{F}_{\mathbf{X}}((x_1,B_2),t_k)}{\partial x_1} &=-h\tilde{f}^a_{(X_{1}, T_{2}) }\left(x_{1}, t_k\right)+hr_2\sum_{\rho=0}^{k-1} \sum_{u_2=0}^{m_2} \frac{\partial\bar{F}_{\mathbf{X}}((x_1,B_2),t_k\vert(B_1,y_{u_2}),t_\rho)}{\partial x_1}\tilde{f}^a_{(X_{2}, T_{1}) }\left(y_{u_2} , t_\rho\right)  \nonumber\\
 &+hr_1\sum_{\rho=0}^{k-1} \sum_{u_1=0}^{m_1} \frac{\partial\bar{F}_{\mathbf{X}}((x_1,B_2),t_k\vert( y_{u_1},B_2),t_\rho)}{\partial x_1}\tilde{f}^a_{(X_{1}, T_{2}) }\left(y_{u_1}, t_\rho\right);\label{ftilda1_espl}\\
 \frac{\partial\bar{F}_{\mathbf{X}}((B_1, x_2),t_k)}{\partial x_2} &= -h\tilde{f}^a_{(X_{2},T_{1})}\left( x_{2}, t_k\right)+hr_2\sum_{\rho=0}^{k-1} \sum_{u_2=0}^{m_2} \frac{\partial\bar{F}_{\mathbf{X}}((B_1, x_2),t_k\vert(B_1,y_{u_2}),t_\rho)}{\partial x_2}\tilde{f}^a_{(X_{2},T_{1})}\left( y_{u_2}, t_\rho\right) \nonumber\\
 &+hr_1\sum_{\rho=0}^{k-1} \sum_{u_1=0}^{m_1} \frac{\partial\bar{F}_{\mathbf{X}}((B_1, x_2),t_k\vert(y_{u_1},B_2),t_\rho)}{\partial x_2}\tilde{f}^a_{(X_{1}, T_{2}) }( y_{u_1} , t_\rho) \label{ftilda2_espl},
 \end{align}}
\end{subequations}
with $m_1,m_2\in\mathbb{N}$.  Here $\tilde{f}^a_{(X_{i}, T_{j}) }\left( y , t\right)$ denotes the approximation of $f^a_{(X_{i}, T_{j}) }\left(y , t\right)$ determined by the time and space discretization and by the truncation of the infinite sums of the space discretization.

Since $f^a_{(X_i,T_j)}(y_{u_i},t_0)=0$, we set $\tilde{f}^a_{(X_i,T_j)}(y_{u_i},t_0)=0$. Moving the term $\tilde{f}^a_{(X_{i}, T_{j})}(x_i,t_k)$ on the left hand side in \eqref{ftilda1}, we obtain the following algorithm to approximate $f^a_{(X_{i}, T_{j})}$ in the knots $\{(y_{u_1},y_{u_2});t_k\}$: \\
\noindent{\bf Step 1}
{\footnotesize{
\begin{eqnarray}
\tilde{f}^a_{(X_{1}, T_{2}) }\left(  y_{u_1} , t_1\right) 
=-\frac{1}{h }\left. \frac{\partial}{\partial x_1}\bar{F}_{\mathbf{X}}((x_1,B_2),t_1)\right\vert_{x_1=y_{u_1}};\nonumber\\
 \tilde{f}_{(X_{2}^{a},T_{1}) }\left( y_{u_2} , t_1\right)\nonumber
 =-\frac{1}{h } \left.\frac{\partial}{\partial x_2}\bar{F}_{\mathbf{X}}((B_1, x_2),t_1)\right\vert_{x_2=y_{u_2}}.
\end{eqnarray}}}
{\bf Step $k\geq 2$}
{\footnotesize{
\begin{eqnarray}
\tilde{f}^a_{(X_{1},T_{2})}\left(  y_{u_1} ,t_k\right) 
&=& - \frac{1}{h}\left. \frac{\partial \bar{F}_{\mathbf{X}}((x_1,B_2),t_k)}{\partial x_1}\right\vert_{x_1=y_{u_1}}+ r_2\sum_{\rho=0}^{k-1}\sum_{v_2=1}^{m_2} \tilde{f}^a_{(X_{2}, T_{1}) }\left( y_{v_2} , t_\rho\right)
  \left.\frac{\partial \bar{F}_{\mathbf{X}}((x_1,B_2),t_k\vert(B_1, y_{v_2}),t_\rho)}{\partial x_1}\right\vert_{x_1=y_{{u_1}}}\nonumber\\
 &+& r_1 \sum_{\rho=0}^{k-1}\sum_{v_1=1}^{m_1} \tilde{f}^a_{(X_{1},T_{2}) }\left(  y_{v_1}, t_\rho\right)  \left.\frac{\partial \bar{F}_{\mathbf{X}}((x_1,B_2),t_k\vert(y_{v_1},B_2) ,t_\rho)}{\partial x_1}\right\vert_{x_1=y_{u_1}};\nonumber\\
  \tilde{f}^a_{(X_{2},T_{1}) }\left( y_{u_2} , t_k\right)\nonumber
 &=& {-}\frac{1}{h}\left.\frac{\partial \bar{F}_{\mathbf{X}}((B_1, x_2),t_k)}{\partial x_2}\right\vert_{x_2=y_{u_2}} + r_2\sum_{\rho=0}^{k-1}\sum_{v_2=1}^{m_2}\tilde{f}^a_{(X_{2}, T_{1})}\left( y_{v_2}, t_\rho\right)
  \left.\frac{\partial \bar{F}_{\mathbf{X}}((B_1, x_2),t_k\vert(B_1, y_{v_2}),t_\rho)}{\partial x_2}\right\vert_{x_2=y_{u_2}}\nonumber\\
 &+& r_1 \sum_{\rho=0}^{k-1}\sum_{v_1=1}^{m_1} \tilde{f}^a_{(X_{1}, T_{2}) }\left(  y_{v_1},t_\rho\right)  \left.\frac{\partial \bar{F}_{\mathbf{X}}((B_1, x_2),t_k\vert(y_{v_1},B_2) ,t_\rho)}{\partial x_2}\right\vert_{x_2=y_{{u_2}}}.\nonumber
\end{eqnarray}}}
At time $t_1$,  $\hat{f}^a_{(X_{i}, T_{j}) }\left(  y_{u_i} , t_1\right)=\tilde{f}^a_{(X_{i}, T_{j}) }\left(  y_{u_i} , t_1\right)$ in each knot $y_{u_i}$.
\begin{remark}
We choose to use the Euler method because it simplifies the notation and it is easy to implement. More efficient schemes, e.g. trapezoidal formula or Gaussian quadrature formulas for infinite integration intervals, can be similarly applied improving the rate of convergence of the error of the algorithm.
\end{remark}
\begin{remark}\label{remark2}
The proposed algorithm can be used to approximate the unknown densities $f^a_{(X_i,T_j)}$ for all types of boundaries. To evaluate the joint density of $(T_1,T_2)$ using \eqref{ProbTTa}-\eqref{ProbTT3}, further steps are needed. First, discretize the integrals in \eqref{ProbTTa}-\eqref{ProbTT3} by means of trapezoidal or Gaussian quadrature formulas. Second, replace $f^a_{(X_i,T_j)}$ by $\widetilde{f}^a_{(X_i,T_j)}$. Finally, since $f_{T_i}$ and $f_{(X_i,T_j)}$ in \eqref{ProbTTa}-\eqref{ProbTT3} are typically unknown, approximate them using the numerical algorithms proposed in \cite{ReviewSac} and \cite{Benedetto}, respectively.
\end{remark}

\section{Convergence of the numerical method for $f^a_{(X_i,T_j)}$}\label{section5}
The error of the algorithm to approximate $f^a_{(X_i,T_j)}$, evaluated in the mesh points $(y_{u_i},t_k)$, is 
\begin{eqnarray}\label{error tot}
E^{(i)}(y_{u_i},t_k):=f^a_{(X_{i},T_{j}) }\left( y_{u_i} , t_k\right)-\tilde{f}^a_{(X_{i},T_{j}) }\left( y_{u_i} , t_k\right), 
\end{eqnarray}
for $k=0,\ldots,N,u_i=1,\ldots,m_i, i=1,2$. Mimicking the analysis of the error in \cite{CMV}, we rewrite  \eqref{error tot} as a sum of two errors.  The first is given by $e^{(i)}_k(y_{u_i}):=f^a_{(X_{i},T_{j}) }\left( y_{u_i} , t_k\right)-\hat{f}^a_{(X_{i},T_{j}) }\left( y_{u_i} , t_k\right)$ and it is due to the time discretization. The second is given by $E^{(i)}_{k,u_i}:  =\hat{f}^a_{(X_{i},T_{j}) }\left( y_{u_i} , t_k\right)-\tilde{f}^a_{(X_{i},T_{j}) }\left( y_{u_i} , t_k\right)$ and it is determined by the spatial discretization and by the truncation introduced at steps $k\geq 2$. 
\begin{lemma}\label{lemma1}
It holds
\begin{subequations}\label{errorE}
{\footnotesize{
\begin{align}
 E^{(1)}_{k,u_1}&=\sum_{\rho=0}^{k-1}\left[- \int_{-\infty}^{B_1}K_{1,k,\rho}((y_{u_1},B_2),(y,B_2)) 
 \hat{f}^a_{(X_{1}, T_{2})}( y , t_\rho) dy\right. - \int_{-\infty}^{B_2} K_{1,k,\rho}((y_{u_1},B_2),(B_1,y))
\hat{f}^a_{(X_{2},T_{1})}\left( y, t_\rho\right)dy \nonumber\\
& +r_1 \sum_{v_1=1}^{m_1} K_{1,k,\rho}((y_{u_1},B_2),(y_{v_1},B_2))\tilde{f}^a_{(X_{1}, T_{2}) }( y_{v_1} , t_\rho)\left.+r_2 \sum_{v_2=1}^{m_2} K_{1,k,\rho}((y_{u_1},B_2),(B_1,y_{v_2}))\tilde{f}^a_{(X_{2},T_{1})}\left( y_{v_2}, t_\rho\right) 
 \right];\label{errorEa}\\
E^{(2)}_{k,u_2}
&=\sum_{\rho=0}^{k-1}\left[ -
\int_{-\infty}^{B_1}K_{2,k,\rho}((B_1, y_{u_2}),(y,B_2)) 
 \hat{f}^a_{(X_{1}, T_{2}) }( y , t_\rho) dy\right.  - \int_{-\infty}^{B_2} K_{2,k,\rho}((B_1, y_{u_2}),(B_1,y))
\hat{f}^a_{(X_{2},T_{1})}\left( y, t_\rho\right)dy \nonumber\\
& +r_1 \sum_{v_1=1}^{m_1} K_{2,k,\rho}((B_1, y_{u_2}),(y_{v_1},B_2))\tilde{f}^a_{(X_{1}, T_{2}) }( y_{v_1} , t_\rho) \left.+r_2 \sum_{v_2=1}^{m_2} K_{2,k,\rho}((B_1, y_{u_2}),(B_1,y_{v_2}))\tilde{f}^a_{(X_{2},T_{1})}\left( y_{v_2}, t_\rho\right)
 \right],\label{errorEb}
\end{align}}}
\end{subequations}
where the kernels are
{\footnotesize{
\begin{eqnarray}
&&K_{1,k,t}((y_{u_1},b),(c,d))=\left. \frac{\partial}{\partial x_1}\left[\bar{F}_{\mathbf{X}}((x_1, b),t_k\vert(c,d),t)-\bar{F}_{\mathbf{X}}((x_1, b),t_{k-1}\vert(c,d),t)
\right]\right|_{x_1=y_{u_1}};\nonumber\\
[-1.5ex]\label{kernel}\\
&&K_{2,k,t}((a,y_{u_2}),(c,d))=\left. \frac{\partial}{\partial x_2}\left[\bar{F}_{\mathbf{X}}((a, x_2),t_k\vert(c,d),t)-\bar{F}_{\mathbf{X}}((a, x_2),t_{k-1}\vert(c,d),t)
\right]\right|_{x_2=y_{u_2}}.\nonumber
\end{eqnarray}}}
To simplify the notation we write $K_{i,k,\rho}$ instead of $K_{i,k,t_\rho}$ when $t=t_\rho$. Here $a,c\in(-\infty,B_1)$ and $b,d\in(-\infty,B_2)$.
\end{lemma}
The proof of Lemma \ref{lemma1} is given in Appendix C. To prove the convergence of the proposed algorithm, we need the following \newline

\noindent\textbf{Regularity conditions} \emph{
For  $i,j=1,2, i\neq j, k=1,\ldots, N, \rho=0,\ldots, k-1, u_i=1,\ldots, m_i$: 
\begin{enumerate}
	\item[(i)] $K_{i,k,\rho}((a,b),(c,y))\hat{f}^a_{(X_{j},T_{i})}\left( y, t_\rho\right)$ and $K_{i,k,\rho}((a,b),(y,d))
 \hat{f}^a_{(X_{i}, T_{j}) }( y , t_\rho)$ are ultimately monotonic (Page 208 in \cite{Davis}) in $y$;
	\item[(ii)] $f^a_{(X_{i},T_{j})}\left( y, t_\rho\right)$  is bounded, it belongs to $L^1$ in $y$ and there exist positive functions $C_{i,1}(y)\in L^1$, $C_{i,2}(y)\in L^1$, with $y\in(-\infty,B_1)$ and $y\in(-\infty,B_2)$, respectively, such that 
	\begin{eqnarray*}
\left|K_{i,k,\rho}((a,b),(y,d))\right|&\leq h C_{i,1}(y); \qquad\qquad \left|K_{i,k,\rho}((a,b),(c,y))\right|&\leq h C_{i,2}(y),
\end{eqnarray*}
and $C_{i,1}(0)$ and $C_{i,2}(0)$ are bounded;
	\item[(iii)]  for  $l=1,2$ 
	{\footnotesize{\begin{eqnarray*}
	&\int_{-\infty}^{B_i-m_i(r_i) r_i}C_{l,i}(y)\left|\hat{f}^a_{(X_{i},T_{j})}\left( y, t_\rho\right)\right|dy\leq \psi_{l,i} r_i,
		\end{eqnarray*}}}
as $r_i\to 0$ and $m_i(r_i)r_i\to \infty$, where $\psi_{l,i}$ are positive constants;
\item[(iv)] for $l=1,2$, there exist constants $Q_{l,i}$ such that
{\footnotesize{\begin{equation*}
	 \left|\int_{-\infty}^{B_i}\frac{\partial}{\partial t}\left[C_{l,i}(y) f^a_{(X_{i},T_{j})}\left( y, t\right)\right]dy\right|\leq Q_{l,i};
\end{equation*}}}
\item[(v)] for $l=1,2$, $\mathbf{z}_1=(y_{u_1},B_2)$ and $\mathbf{z}_2=(B_1,y_{u_2})$,
{\footnotesize{
\begin{eqnarray*}
&&\frac{\partial}{\partial t }\left[\bar{F}_{\mathbf{X}}(\mathbf{z}_l,t_\rho|(B_1,y),t)f^a_{(X_2,T_{1})}(y,t)\right]\vert_{t=\tau}\in L^1 \textrm{ in  }  y\in (-\infty,B_2);\\
&&\frac{\partial}{\partial t}\left[ \bar{F}_{\mathbf{X}}(\mathbf{z}_l,t_\rho|(y,B_2),t)f^a_{(X_1,T_{2})}(y,t)\right]\vert_{t=\tau}\in L^1 \textrm{ in } y\in (-\infty,B_1);\\
&&\frac{\partial}{\partial y_{u_l}}\frac{\partial}{\partial t }\left[\bar{F}_{\mathbf{X}}(\mathbf{z_l},t_\rho|(B_1,y),t)f^a_{(X_2,T_{1})}(y,t)\right]\vert_{t=\tau}\in L^1 \textrm{ in  } y\in (-\infty,B_2);\\
&&\frac{\partial}{\partial y_{u_l}}\frac{\partial}{\partial t}\left[ \bar{F}_{\mathbf{X}}(\mathbf{z}_l,t_\rho|(y,B_2),t)f^a_{(X_1,T_{2})}(y,t)\right]\vert_{t=\tau}\in L^1 \textrm{ in } y\in (-\infty,B_1).
\end{eqnarray*}}}
\end{enumerate}}
\noindent The following theorem gives the convergence of the proposed algorithm.
\begin {theorem} \label{error}
Denote $r=\max(r_1,r_2)$. If regularity conditions (i)--(v) are satisfied, then 
\begin{eqnarray*}
&&\max\left\{|E^{(i)}(y_{u_i},t_k)|, y_{u_i}=B_i-u_i r_i, t_k=h k, u_i\in\mathbb{N}, k=0,\ldots, N\right\} = O(h)+O(r).
\end{eqnarray*}
\end{theorem}
The proof of Theorem \ref{error} is given in Appendix D.
\begin{remark}
The numerical approximation $\hat f^a_{(X_j,T_i)}(y,t_\rho)$ can be rewritten as a function of $f^a_{(X_j,T_i)}(y,t_\rho)$ and $K_{i,k,\rho}((a,b),(c,d))$, as shown in Remark 2 in \cite{CMV}. Therefore, assumptions (i) and (iii) are in fact assumptions on $f^a_{(X_j,T_i)}$ and $K_{i,k,\rho}((a,b),(c,d))$.
\end{remark}
\begin{remark}{
Theorem 3 holds for any system of integral equations satisfying regularity conditions (i)-(v). We explicitly list them to simplify the proof of the theorem but some of them are always fulfilled by bivariate diffusion processes. In particular we easily note that:
\begin{itemize}
\item[(ii)] 
\begin{itemize}
\item The function $f_{(X_{i}^a,T_{j})}\left( y, t\right)$  is the pdf of a diffusion process and thus it is bounded and it belongs to $L^1$ in both $y$ and $t$.
\item The function $K_{i,k,\rho}((a,b),(y,d))$ is the difference between two bivariate functions. Each
 function is a probability distribution with respect to one variable and a probability density with respect to the other. Furthermore the functions in the difference are the same but computed at different times. Since the process is a diffusion, this difference is bounded. 
\end{itemize}
\item[(iii)] {This condition is not restrictive: since $\lim_{m_i(r_i)r_i\to\infty} \int_{-\infty}^{B_i-m_i(r_i)r_i} C_{l,i}(y)\left|\hat{f}^a_{(X_ i,T_{j})}\left( y, t_\rho\right)\right|dy=0$ $  \forall l, i, t_\rho, y$, it is always possible to choose $m_i(r_i)r_i$ large enough to have condition (iii).}
\item[(iv)] $C_{l,i}(y)$ does not depend upon $t$ and thus
\begin{eqnarray*}
 \left|\int_{-\infty}^{B_i}\frac{\partial}{\partial t}\left[C_{l,i}(y) f^a_{(X_i,T_{j})}\left( y, t_\rho\right)\right]dy\right|=  \left|\int_{-\infty}^{B_i}C_{l,i}(y) \frac{\partial}{\partial t}\left[f^a_{(X_{i},
T_{j})}\left( y, t_\rho\right)\right]dy\right|\leq {\sup_y |\frac{\partial}{\partial t} f^a_{(X_i,T_j)}| \int_{-\infty}^{B_i}\left|C_{l,i}(y) \right|dy \leq Q_{l,i}},
\end{eqnarray*}
since $\frac{\partial}{\partial t} f^a_{(X_i,T_j)}$ is bounded being the derivative of a density of a diffusion process.
\item[(v)] If we differentiate \eqref{Volterra} with respect to $t$ or with respect to both $t$ and $x_i$, than the left hand side of \eqref{Volterra1} would be finite, since the function  $\bar F_{\mathbf{X}}(\mathbf{x},t)$ is a bivariate probability distribution of a diffusion process satisfying Kolmogorov equation. Then the corresponding right hand side should also be finite and thus assumption (v) is fulfilled.
\end{itemize}
Assumption (i) does not have an immediate interpretation.However it is verified by many important diffusion processes such as the Wiener and the Ornstein Uhlenbeck processes. Moreover, it {holds for stationary processes for large values of $t$ since their distribution does not depend upon the initial condition.}.}
\end{remark}

\section{{Special case:} bivariate Wiener process}\label{section6}
Consider a bivariate process $\mathbf{X}$ solving  \eqref{diff} with $t_0=0$, constant drift 
$\bm \mu(\mathbf{X}(t))=(\mu_1,\mu_2)\in \mathbb{R}^2$ and positive-definite diffusion matrix
\[
\bm \Sigma \left(\bm X(t)\right)=\left( 
\begin{array}{cc}
\sigma _{1} & 0\\ 
\rho\sigma_2 & \sigma _{2}\sqrt{1-\rho^2}%
\end{array}%
\right),
\]
for $\sigma_1,\sigma_2>0, \rho \in (-1,1)$. Then $\bm X$ is a bivariate Wiener process with drift $(\mu_1,\mu_2)$ and covariance matrix
\[
\widetilde{\bm\Sigma}=\left( 
\begin{array}{cc}
\sigma _{1}^2 & \rho\sigma_1\sigma_2\\ 
\rho\sigma_1\sigma_2 & \sigma _{2}^2
\end{array}%
\right),
\]
{and the densities $f_{\mathbf{X}}, f^a_{X_i}$ and $f_{T_i}$ are available \cite{coxMiller} in closed form. The unknown density $f^a_{\mathbf{X}}$ is a solution of the two-dimensional Kolmogorov forward equation}
\begin{equation*}
\label{FokkerPlanck} \frac{\partial f^a_{\mathbf{X}}(\mathbf{x},t) }{\partial t}=
\frac{\sigma _{1}^{2}}{2}\frac{\partial ^{2}f^a_{\mathbf{X}}( \mathbf{%
x},t) }{\partial x_{1}^{2}}+\frac{\sigma _{2}^{2}}{2}\frac{\partial
^{2}f^a_{\mathbf{X}}( \mathbf{x},t)}{\partial x_{2}^{2}}+\sigma _{1}\sigma _{2}\rho \frac{\partial ^{2}f^a_{\mathbf{X}}( \mathbf{x},t)}{\partial x_{1}\partial x_{2}}-\mu _{1}\frac{\partial f^a_{\mathbf{X}}( \mathbf{x},t) }{%
\partial x_{1}}-\mu _{2}\frac{\partial f^a_{\mathbf{X}}( \mathbf{x}%
,t) }{\partial x_{2}}, 
\end{equation*}%
with initial, boundary and absorbing conditions given by
\begin{eqnarray}
\lim_{t\rightarrow 0}f^a_{\mathbf{X}}\left( \mathbf{x},t\right) &=&\delta
\left( x_1-x_{01}\right) \delta \left( x_2-x_{02}\right);  \label{initial} \\
\lim_{x_{1}\rightarrow -\infty }f^a_{\mathbf{X}}\left( \mathbf{x},t\right)
&=&\lim_{x_{2}\rightarrow -\infty }f^a_{\mathbf{X}}\left( \mathbf{x}%
,t\right) =0; \label{bound} \\
\left. f^a_{\mathbf{X}}\left( \mathbf{x},t\right) \right\vert
_{x_{1}=B_{1}} &=&\left. f^a_{\mathbf{X}}\left( \mathbf{x},t\right)
\right\vert _{x_{2}=B_{2}}=0, \label{abs}
\end{eqnarray}
respectively. The solution provided in \cite{Domine} does not fulfil  \eqref{initial} when $(\mu_1,\mu_2)\neq (0,0)$. Mimicking their proof, we noted that the normalizing factor
{\footnotesize{\begin{equation}\label{constant}
\exp \left(-\frac{\left( \mu _{2}\rho \sigma _{1}\sigma _{2}-\mu _{1}\sigma^2
_{2}\right) B_{1}+\left( \mu _{1}\rho \sigma _{1}\sigma _{2}-\mu _{2}\sigma^2
_{1}\right) B_{2}}{\left( 1-\rho ^{2}\right) \sigma _{1}^{2}\sigma _{2}^{2}}\right)
\end{equation}}}
\noindent is missing. Since it equals 0 when $(\mu_1,\mu_2)=(0,0)$, the results in \cite{Domine} are correct for the driftless case and correspond to those provided in \cite{Metzler}. Summarizing these observations, in presence of drift we have
\begin{lemma}\label{lemmafaj} 
The joint pdf of $\mathbf{X}$ up to time $T$ is 
{\footnotesize{\begin{equation} \label{finalfa} 
f^a_{\mathbf{X}}(\mathbf{x},t)=\frac{2}{\alpha K_{3}t}\exp\left(
K_{1}(B_{1}-x_{01})+K_{2}(B_{2}-x_{02}) -\frac{\sigma _{1}^{2}\mu _{2}^{2}-2\mu _{1}\mu
_{2}\sigma _{1}\sigma _{2}\rho +\sigma _{2}^{2}\mu _{1}^{2}}{2K_{3}^{2}}t-%
\frac{\bar{r}^{2}+\bar{r}_{0}^{2}}{2K_{3}^{2}t}\right) H(\bar{r},\bar{r}%
_{0},\phi ,\phi _{0},t),  
\end{equation}}}
where $\bar r:= \bar r(\bm x) \in (0,\infty), \phi:=\phi(\bm x) \in (0,\alpha)$ and
{\footnotesize{\begin{eqnarray*}
\bar{r}&=&\sqrt{\sigma _{1}^{2}(B_{2}-x_{2})^{2}+\sigma
_{2}^{2}(B_{1}-x_{1})^{2}-2\sigma _{1}\sigma
_{2}\rho(B_{1}-x_{1})(B_{2}-x_{2})}; \qquad\qquad \bar{r}_0=\bar{r}_{\vert_{\bm x=\bm x_0}};\\
\bar{r}\cos(\phi)&=&\sigma_2(B_1-x_1)-\sigma_1\rho (B_2-x_2),\quad 
\bar{r}\sin(\phi)=\sigma_1\sqrt{1-\rho^2}(B_2-x_2); \qquad\qquad \phi_0=\phi_{\vert_{\bm x=\bm x_0}};\\
K_{1} &=&\frac{\sigma _{2}\mu _{1}-\sigma _{1}\mu _{2}\rho }{\sigma
_{1}^{2}\sigma _{2}(1-\rho ^{2})},\qquad K_{2}=\frac{\sigma _{1}\mu
_{2}-\sigma _{2}\mu _{1}\rho }{\sigma _{1}\sigma _{2}^{2}(1-\rho ^{2})},%
\qquad K_{3}=\sigma _{1}\sigma _{2}\sqrt{1-\rho ^{2}} ;\\
\alpha &=&\arctan \left( -\frac{\sqrt{1-\rho ^{2}}}{\rho }\right) \in (0,\pi); \\
H(\bar{r},\bar{r}_{0},\phi ,\phi _{0},t) &=&\sum_{n=1}^{\infty }\sin \left(\frac{%
n\pi \phi _{0}}{\alpha }\right)\sin \left(\frac{n\pi \phi }{\alpha }\right)I_{n\pi /\alpha
}\left( \frac{\bar{r}\bar{r}_{0}}{K_{3}^{2}t}\right).
\end{eqnarray*}}}
\end{lemma}
{Here $I_\rho(x)$ denotes the modified Bessel function of the first kind \cite{Watson} and }$\bar r$ and $\phi$ are functions of $\bm x$ obtained through a suitable change of variables \cite{Domine}. We use them instead of $\bm x$ to simplify the notation. \newline
From the definition of conditional density $f^a_{X_2|X_1}$, \eqref{finalfa} and $f^a_{X_1}$ given in \cite{ReviewSac}, it follows
\begin{corollary}\label{corx}
\label{corfamarg} The conditional density $
f^a_{X_{i}\vert X_{j}}(x_{i},t\vert x_{j},t)$, for $i,j=1,2;$ $i\neq j$ is  
{\footnotesize{\begin{equation}\label{facondXWiener}
f^a_{X_{i}\vert X_{j}}(x_{i},t\vert x_{j},t)=\frac{\frac{2\sigma _{j}\sqrt{2\pi t}}{%
\alpha K_{3}t}\exp \left( -K_{i}\left[ \frac{\sigma _{i}}{\sigma _{j}}%
(x_{j}-x_{0j})\rho -(x_{i}-x_{0i})+ N_j t\right]-\frac{\bar{r}^2+\bar{r}_0^2-(x_{j}-x_{0j})^{2}\sigma _{i}^{2}(1-\rho ^{2})}{2K_{3}^{2}t}
\right) }{\left[ 1-\exp \left( \frac{2(B_{j}-x_{0j})(x_{j}-B_{j})}{\sigma
_{j}^{2}t}\right) \right]}H(\bar{r},\bar{r}_{0},\phi ,\phi _{0},t),
\end{equation}}}
with
{\footnotesize{\[
N_{1}=\frac{\sigma _{1}\mu _{2}-\sigma _{2}\mu _{1}\rho }{2\sigma _{1}},
\qquad N_{2}=\frac{\sigma _{2}\mu _{1}-\sigma _{1}\mu _{2}\rho }{2\sigma _{2}}.\]}}
\end{corollary}

\begin{lemma}\label{corfa} 
The conditional density $f^a_{X_{i}\vert T_{j}}(x_i\vert t)$ is   
{\footnotesize{
\begin{eqnarray}
f^a_{X_{i}\vert T_{j}}(x_i\vert t)&=&\frac{\sigma _{j}\pi \sqrt{2\pi t}}{\alpha
^{2}(B_{i}-x_i)(B_{j}-x_{0j})}\exp \left( -K_{i}\left[ \frac{\sigma _{i}}{%
\sigma _{j}}(B_{j}-x_{0j})\rho -(x_i-x_{0i}) + N_j t\right] \right) \nonumber\\[-1.5ex] \label{facondXTWiener}\\[-1.5ex]
&\times &\exp \left(-\frac{[\rho \sigma
_{i}(B_{j}-x_{0j})-\sigma _{j}(B_{i}-x_{0i})]^{2}+\sigma
_{j}^{2}(B_{i}-x_i)^{2}}{2K_{3}^{2}t}\right) G_{ij}(\bar{r}_{0},\phi
_{0},x_i,t),\nonumber
\end{eqnarray}}}
where 
{\footnotesize{\[
G_{ij}(\bar{r}_{0},\phi _{0},x_{i},t)=\sum_{n=1}^{\infty }\delta _{i}n\sin \left( 
\frac{n\pi \phi _{0}}{\alpha }\right) I_{\frac{n\pi }{\alpha }}\left( \frac{%
\sigma _{j}(B_{i}-x_{i})\bar{r}_{0}}{K_{3}^{2}t}\right),
\]}}
with $\delta_{1}=1$ and $\delta _{2}=(-1)^{n+1}$ for $i,j=1,2, i\neq j$.
\end{lemma}
The proof of Lemma \ref{corfa} is given in Appendix E. From Lemma \ref{corfa} and $f_{T_j}$, it follows 
\begin{corollary}
The joint density $f_{(X_i^a,T_j)}$ is given by
{\footnotesize{\begin{eqnarray}
f^a_{(X_{i},T_{j})}(x_i, t)&=&\frac{1}{\alpha
^{2}(B_{i}-x_i)t}\exp \left( -K_{i}\left[ \frac{\sigma _{i}}{%
\sigma _{j}}(B_{j}-x_{0j})\rho -(x_i-x_{0i}) + N_j t\right] -\frac{(B_j-x_{0j}-\mu_j t)^2}{2\sigma^2_j t} \right) \nonumber\\[-1.5ex] \label{facondXTWiener}\\[-1.5ex]
&\times &\exp \left(-\frac{[\rho \sigma
_{i}(B_{j}-x_{0j})-\sigma _{j}(B_{i}-x_{0i})]^{2}+\sigma
_{j}^{2}(B_{i}-x_i)^{2}}{2K_{3}^{2}t}\right) G_{ij}(\bar{r}_{0},\phi
_{0},x_i,t).\nonumber
\end{eqnarray}}}
\end{corollary}
The same result can be obtained by solving the system \eqref{Volterra} (calculations not shown). Since the dependence between $X_1$ and $X_2$ is only determined by the diffusion matrix and not by the drift term, 
the slowest component becomes independent on the fastest one after time $T$ in presence of absorbing boundaries. That is $f_{T_i}(t_i|(B_j,x_i),t_j)=f_{T_i}(t_i|x_i,t_j)$. Thus the joint pdf of $(T_1,T_2)$ for absorbing or killing boundaries is the same. In particular, it holds
\begin{theorem}\label{theoW}
In presence of either absorbing or killing boundaries, the joint density of $(T_1,T_2)$ is
{\footnotesize{\begin{eqnarray}
f_{(T_1,T_2)}(t_1,t_2)&=&\sum^2_{i,j=1; i\neq j}
\frac{\sqrt{2\pi }}{2\alpha
^{2}\sigma _{j}t_i\sqrt{(t_j-t_i)^{3}}}\exp \left( K_{i}(B_{i}-x_{0i})-K_{j}x_{0j}+\frac{\mu
_{j}B_{j}}{\sigma _{j}^{2}} -K_{i}N_{j}t_i  -\frac{\mu _{j}^{2}t_j}{2\sigma _{j}^{2}}\right)\nonumber\\[-1.5ex] \label{t1t2Wiener}\\[-1.5ex]
 &\times& \int_{-\infty
}^{B_{j}}\exp \left( -K_{i}\rho 
\frac{\sigma _{i}}{\sigma _{j}}x_{j} \right)\exp \left( -\frac{\bar{r}_{0}^{2}+\sigma
_{i}^{2}(B_{j}-x_{j})^{2}}{2t_iK_{3}^{2}}-\frac{(B_{j}-x_{j})^{2}}{2\sigma
_{j}^{2}(t_j-t_i)}\right) G_{ji}(\bar{r}%
_{0},\phi _{0},x_j,t_i)dx_{j}.\nonumber
\end{eqnarray}}}
\end{theorem}
\begin{proof}
It follows by plugging $f_{T_i}$, given in \cite{coxMiller}, and (\ref{facondXTWiener}) into (\ref{ProbTT}).\qed \end{proof} 
\begin{remark}
If $(\mu_1,\mu_2)=(0,0)$, mimicking the proof in \cite{Domine, Metzler}, the joint density \eqref{t1t2Wiener} becomes 
{\footnotesize{\begin{eqnarray}\label{densmu0}
f_{(T_1,T_2)}(t_1,t_2)&=&\sum^2_{i,j=1; i\neq j}
\frac{\pi \sqrt{1-\rho ^{2}}\exp \left( -\frac{\bar{r}_{0}^{2}[t_j+t_i(1-2\rho ^{2})]%
}{4K_{3}^{2}t_i(t_j-t_i\rho ^{2})}\right)}{2\alpha ^{2}\sqrt{%
t_i(t_j-t_i\rho ^{2})}(t_j-t_i)}\sum_{n=1}^{\infty }\delta_j n\sin\left( 
\frac{n\pi \phi _{0}}{\alpha }\right)I_{\frac{n\pi }{2\alpha }}\left( \frac{\bar{r}%
_{0}^{2}(t_j-t_i)}{4K_{3}^{2}t_i(t_j-t_i\rho ^{2})}\right); \\
\nonumber f_{(T_1,T_2)}(t,t)&=&
\begin{cases}
0  & \textrm{if } \rho \in (-1,0)\\
\infty & \textrm{if } \rho \in (0,1)\\
\frac{(B_1-x_{01})(B_2-x_{02})}{2\pi\sigma_1\sigma_2 t^3 }\exp\left\{-\frac{\sigma_2^2(B_1-x_{01})^2+\sigma_1^2(B_2-x_{02})^2}{2\sigma^2_1\sigma_2^2 t}\right\}
 & \textrm{if } \rho =0\\
\end{cases}.
\end{eqnarray}}}
\end{remark}
When $\rho=0$, the processes are not correlated and thus $f_{(T_1,T_2)}=f_{T_1}f_{T_2}$, as expected. For the same reason, $f^a_{\mathbf{X}}=f^a_{X_1}f^a_{X_2} , f^a_{X_j|X_i}=f^a_{X_j}$ and $f^a_{X_j|T_i}=f^a_{X_j}$  (calculations not reported).

\begin{remark}
To compare \eqref{densmu0} with the corresponding expression in \cite{Iyengar}, we set $s=t_1<t=t_2$ and $\tilde{r}_{0}=\bar{r}_{0}/K_{3}$, because different transformations are used. Since
{\footnotesize{\[
\sqrt{1-\rho ^{2}}=\sin(\alpha), \qquad \rho ^{2}=\cos ^{2}(\alpha), \qquad 2(t-s\rho^2) = (t-s)+(t-s\cos(2\alpha) ),
\]}}
we obtain
{\footnotesize{
\[
f_{(T_1,T_2)}(s,t)=\frac{\pi \sin(\alpha)\exp \left( -\frac{\tilde{r}_{0}^{2}(t-s\cos(2\alpha) )}{2s[(t-s)+(t-s\cos(2\alpha) )]}\right) }{2\alpha ^{2}\sqrt{s(t-s\cos ^{2}(\alpha)
)}(t-s)} \sum_{n=1}^{\infty }(-1)^{n+1}n\sin \left(\frac{n\pi \phi _{0}}{\alpha }\right
)I_{\frac{n\pi }{2\alpha }}\left( \frac{\tilde{r}_{0}^{2}(t-s)}{%
2s[(t-s)+(t-s\cos(2\alpha))]}\right), 
\]}}
for $s<t$.  The result differs from that in \cite{Iyengar}, as already discussed in \cite{Domine}, and agrees with that in \cite{Domine,Metzler}. Since the error disappears when $t_1\to t_2$, the expression $f_{(T_1,T_2)}(t,t)$ in \cite{Iyengar} is correct.
\end{remark}
In Fig. \ref{FigW} we report the theoretical joint density and contour plots of $(T_1,T_2)$ for two correlated Wiener processes in presence of either absorbing or killing boundaries when $\sigma_1=\sigma_2=1$ and $\rho=0.5$. We consider a symmetric and a non symmetric process by choosing null drifts and $B_1=B_2=1$ (upper panels), and $\mu_1=1, \mu_2=1.5$ and $B_1=B_2=10$ (lower panels), respectively. Note that $f_{(T_1,T_2)}$ is given by \eqref{densmu0} and  \eqref{t1t2Wiener}, respectively. For symmetric processes with relatively small variances, the probability mass of $f_{(T_1,T_2)}$ is concentrated in the area close to the diagonal $t_1=t_2$. For non symmetric processes, the probability mass is concentrated around the means of the FPTs, i.e. $\mathbb{E}[T_1]=10$ and $\mathbb{E}[T_2]=6.67$, and it is spread out according to their variances, i.e. $\textrm{Var}(T_1)=12.5$ and $\textrm{Var}(T_2)=3.70$ for the considered parameter choice. In presence of crossing boundaries, a closed expression for $f_{(T_1,T_2)}$ using \eqref{ProbTT3} cannot be derived since an analytical expression for $f_{(X_j,T_i)}$ is not available. However $f_{(T_1,T_2)}$ can be numerically approximated as described in Remark \ref{remark2}.
\begin{figure}[t!] 
\centering 
\includegraphics[scale=.8]{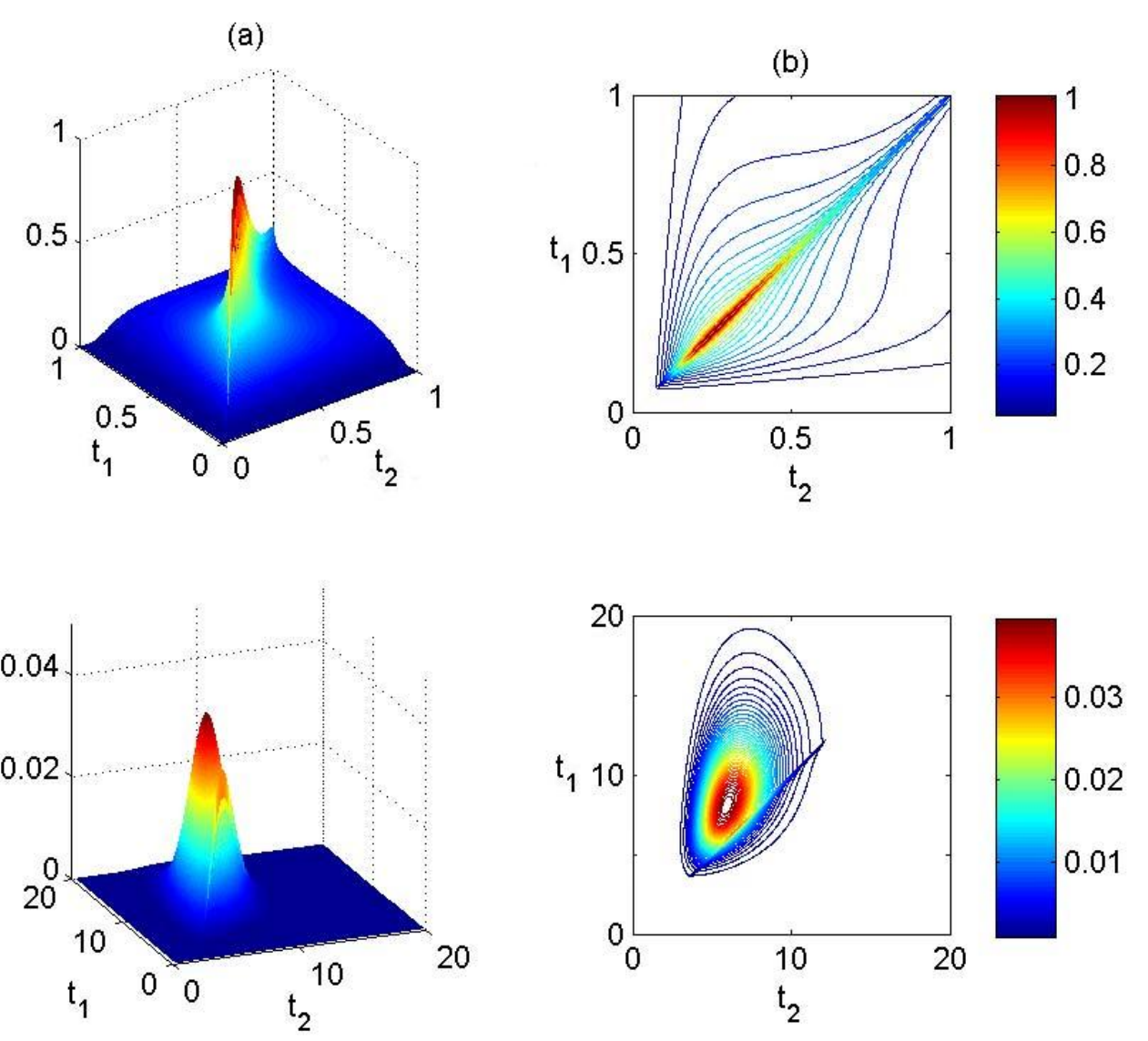}
\caption{Theoretical joint densities and contour plots of $(T_1,T_2)$ for two-dimensional correlated Wiener processes in presence of either absorbing or killing boundaries when $\sigma_1=\sigma_2=1$ and $\rho=0.5$. Top panels: $\mu_1=\mu_2=0, B_1=B_2=1 $. Bottom panels: $\mu_1=1, \mu_2=1.5, B_1=B_2=10$ . Panel (a): $f_{(T_1,T_2)}$. Panel (b): contour plots of $(T_1,T_2)$.}
\label{FigW}
\end{figure}

\section{Examples}\label{section7}
Here we discuss some examples  to illustrate the proposed numerical algorithm. First we compare theoretical results of the previous Section on the bivariate Wiener process with those obtained by means of our numerical approach. Then we describe a modeling problem in the neuroscience framework and we use the proposed algorithm to determine quantities of interest for such models.

\subsection{Bivariate Wiener process} 
We numerically illustrate the convergence of the algorithm and discuss its performance in the case of a bivariate, correlated and symmetric Wiener process with $\mu_1=\mu_2=0$, $\sigma_1=\sigma_2=1$ and $\rho=0.5$ in presence of either killing or absorbing boundaries $B_1=B_2=1$. We consider discretization steps given by $h,r=0.01, 0.05, 0.1$ and $0.2$.  We compare the values of $f^a_{(X_i,T_j)}$ and $\widetilde{f}^a_{(X_i,T_j)}$ by considering 
both the maximum of the absolute value of the error $|E^{(i)}(y_{u_i},t_k)|$ and its mean squared error $\textrm{MSE}(f^a_{(X_i,T_j)})$, defined by
\[
\textrm{MSE}(f^a_{(X_i,T_j)})=\frac{ \sum_{y_{u_i},t_k}\left( f^a_{(X_i,T_j)}(y_{u_i},t_k)-\widetilde{f}^a_{(X_i,T_j)}(y_{u_i},t_k)\right)^2}{(m_i+1) N}.
\]
Since the process $\mathbf{X}$ is symmetric, the errors $E^{(1)}(y_{u_1},t_k)$ and $E^{(2)}(y_{u_1},t_k)$ are similar and thus we only discuss the first. 
The maximum absolute error of the algorithm and its mean squared error for $r=1$ are reported in Table \ref{tableW}. As expected from Theorem \ref{error}, $\max |E^{(i)}(y_{u_i},t_k)|$ is no larger than $O(h)+O(r)$. 
\begin{remark}
The maximum error of the proposed approximation appears in correspondence of knots proximal to the boundary. Excessively small space integration steps are discouraged by this fact. We suggest to choose a smaller adaptive step for values far from the boundary and a larger one for values near the boundary.
\end{remark}

\begin{table}[t]
\centering
\begin{tabular}{|c|c|c|}
\hline
$h$ &  $\max |E^{(1)}(y_{u_1},t_2)|$ & \textrm{MSE}$(f^a_{(X_1,T_2)})$\\\hline
0.01& 0.0335&   0.0028\\
0.05& 0.1779& 0.0188 \\
0.1& 0.2505& 0.0338  \\  
0.2& 0.3135&    0.0517 \\
\hline
\end{tabular}
\caption{Maximum absolute error $\max |E^{(1)}(f^a_{(X_1,T_2)})|$ and mean squared error of the algorithm for a bivariate, correlated and symmetric Wiener process for $\mu_1=\mu_2=0, \sigma_1=\sigma_2=1$ and $\rho=0.5$. Different values of the time discretization step $h$ when the space discretization step $r$ is $0.1$. Since the process is symmetric, the performance of the error with respect to $f^a_{(X_2,T_1)}$ is similar and thus not reported.}
\label{tableW}
\end{table}

\subsection{Bivariate Ornstein-Uhlenbeck as model for neural spiking activity}
The membrane of a neuron is characterized by a difference of potential between its internal and external part. This difference changes in time due to the arrival of excitatory and inhibitory inputs from other surrounding neurons. When the membrane depolarization attains a specific value, called threshold value, an electrical output, called spike, is released. After a spike the membrane potential spontaneously resets to a resting value and the membrane potential evolution restarts. Popular models for the neuronal dynamics are the so called Leaky-Integrate and Fire (LIF) models (see \cite{ReviewSac} for a review),  which identify firing times with FPTs of the process through a boundary. The OU process is probably the most commonly used LIF model. This is because it combines good levels of mathematical tractability, experimental fit and biological motivations. An extension to a LIF model describing the joint behavior of a set of interacting neurons has been recently formalized in \citep{TSJ}. There multivariate OU processes and their FPTs are obtained as diffusion limits of multivariate jump processes and their FPTs. The parameters of the limiting processes are a mixture of terms due to inputs which are specific for each neuron and other which are common to a group of neurons. In the bivariate case, the resulting OU process satisfies \eqref{diff} with
{{\begin{equation*}
\mathbb{\mu}(\mathbf{X}(t))=\left(
\begin{array}{c} 
\mu_1-\frac{X_1(t)}{\theta}\\ 
\mu_2-\frac{X_2(t)}{\theta}
\end{array}\right) 
\mbox{, } \qquad
\bm \Sigma\left(\bm X(t)\right)=\bm \Sigma=\left(\begin{array}{ll} 
\sigma_{11}& \sigma_{12}\\
 \sigma_{12}& \sigma_{22}
\end{array} \right),
\end{equation*}}}
for $\mu_i\in \mathbb{R}, \sigma_{ij}>0, i,j=1,2,\sigma_{12}\in \mathbb{R}$ and $\bm \Sigma$ positive-definite matrix. The interesting quantity is the joint FPT distribution which allows to study the behavior of the two dependent neurons.
\begin{figure}[t!]
\centering 
\includegraphics[scale=.8]{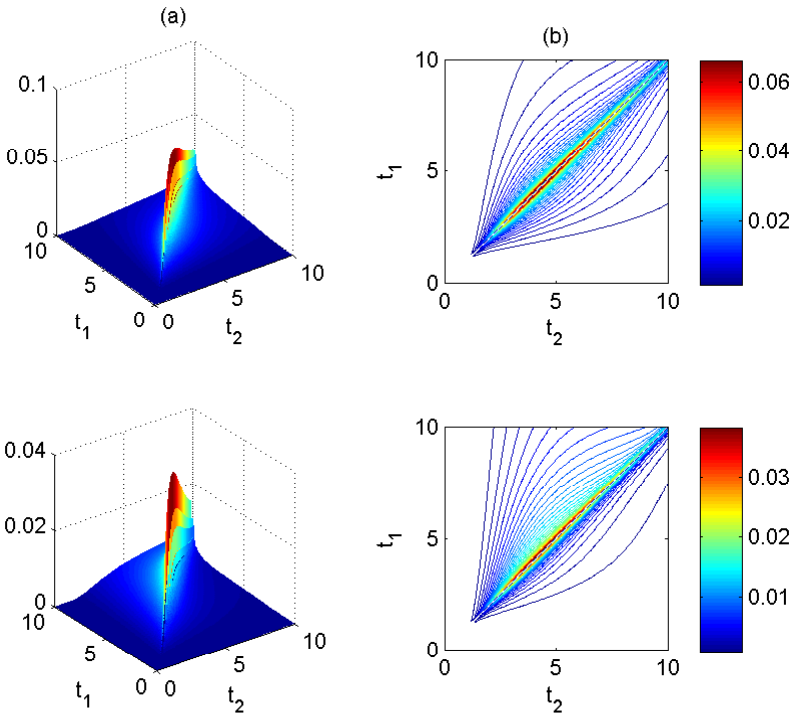}
\caption{Approximated joint densities and contour plots of $(T_1,T_2)$ for two-dimensional correlated OU processes in presence of absorbing boundaries, when $\sigma_{11}=\sigma_{22}=2, \sigma_{12}=1$ and $B_1=B_2=10 $. Top panels: $\mu_1=\mu_2=1.5 $. Bottom panels: $\mu_1=0.95, \mu_2=1.5 $. Panel (a): $\hat f_{(T_1,T_2)}$. Panel (b): contour plots of $(T_1,T_2)$.}
\label{figOU}
\end{figure}

Combining the proposed algorithm and that of $f_{T_j}$, as described in Remark \ref{remark2}, we are able to approximate the joint density of the spike times, i.e. the joint FPT density of the OU process. In Fig. \ref{figOU} we report the approximated joint density and the contour plots of $(T_1,T_2)$ for two symmetric correlated OU with $\mu_1=\mu_2=1.5$ (top panels) and for two non symmetric correlated OU with $\mu_1=0.95$ and $\mu_2=1.5$ (bottom panels). In both cases $\theta=10, \sigma_{12}=1, \sigma_{ii}=2$ and $B_i=10 $, for $i=1,2$. The parameter values are chosen according to biologically acceptable ranges. We omit to report the units to simplify the reading. Looking at the figure we recognize some neuronal features caught by the model. In the first case both the asymptotic membrane potential means $\mu_i\theta$ are above the boundaries $B_i, i=1,2$. Not surprisingly the probability mass of $f_{(T_1,T_2)}$ for the symmetric OU is concentrated around the diagonal $t_1=t_2$, suggesting that the epochs of the passage times are similar. Hence these parameter values can be used to model instances characterized by almost synchronous spikes. For the asymmetric OU, the first component has asymptotic mean $\mu_1\theta$ below $B_1$, and thus its FPT is determined by the noise. As a consequence, the probabilistic mass of $f_{(T_1,T_2)}$ is concentrated in the region $t_1>t_2$ and the firing rates of two neurons differ. Thanks to the algorithm to approximate the joint firing time distribution, further analyses can be done, but we postpone them to a future paper. Extensions of our method to higher dimensions could allow to describe the firing activity of groups of neurons in a network, as modeled in \citep{TSJ}. 
\section{Conclusion}
For general {bivariate} diffusion processes, explicit expressions of the joint FPT density are not available. For killing, absorbing or crossing boundaries, we show how $f_{(T_1,T_2)}$ depends on the unknown density $f^a_{(X_i,T_j)}$. To avoid the prohibitive computational efforts required to approximate $f^a_{(X_i,T_j)}$ via Monte Carlo simulations \cite{Zhou}, we suggest to use our ad hoc numerical method, whose linear convergence in time and space has been proved. 

The choice of considering constant boundaries is not a shortcoming. Indeed both the theoretical results in Section \ref{section3} and the numerical algorithm can be easily  extended to the case of time dependent boundaries.

In presence of either absorbing or killing boundaries, the joint FPT density can be numerically obtained combining our algorithm for $f^a_{(X_i,T_j)}$ with that for $f_{T_j}$ \cite{ReviewSac}. Theoretical expressions for $f_{(T_1,T_2)}$ and other quantities of interest, such as $f^a_{(X_i,T_j)}$ and $f^a_{\mathbf{X}}$, i.e. the transition density of the process constrained to be below the boundaries, are available for a bivariate correlated Wiener process with drift. Here we derive them, correcting  the formulas in \cite{Domine}.  In presence of crossing boundaries, the joint FPT density can be numerically evaluated combining our algorithm with that proposed for $f_{(X_i, T_j)}$ for bivariate Gaussian diffusion processes \cite{Benedetto}.   

Our approach can be extended to other first passage problems. Consider for example a bivariate process whose component is reset whenever it attains its boundary, and then both components evolve until the first crossing of the slowest component. Also in this scenario, which is  commonly used for modeling neural spikes activity in neuroscience \cite{STZ}, the joint FPT density  depends on both $f^a_{(X_i, T_j)}$ and $f_{(X_j, T_i)}$.  

Finally we emphasize that our approach and results can be extended to the $k$-dimensional case. The joint FPT density can be obtained mimicking Theorem \ref{theo}, where the $k$-dimensional extension of $f^a_{(X_i, T_j)}$ can be obtained as solution of a system of $k$ Volterra-Fredholm first kind integral equations. An alternative approach would be to obtain it by first solving a $k$-dimensional Kolmogorov forward equation in presence of absorbing boundaries, generalizing Lemma \ref{lemmafaj}, and then doing the limit for $X_j\to B_j$. Since analytical solutions will be difficult to obtain, one could approximate them by extending our algorithm, starting from the extended Theorem \ref{Volt} and mimicking Section \ref{section4}. New numerical issues may arise and we postpone this study to a future work.

\section*{Acknowledgments}
L.S. and C.Z. were supported by University of Torino (local grant 2013, ZUCCRILO 13: Modelli stocastici e statistici per le applicazioni) and by project AMALFI -
Advanced Methodologies for the AnaLysis and management of the
Future Internet (Universit\`{a} di Torino/Compagnia di San
Paolo). 

\appendix

\section*{Appendix A: Proof of Theorem \ref{theo}}\label{appA}
In presence of absorbing boundary $\bm B$, we have
\begin{eqnarray}
&&\mathbb{P}(T_1<t_1,T_2<t_2)\nonumber\\
&&{\qquad= \int_{t_0}^{t_1}\mathbb{P}(T_{1}<t_{1},T_{2}<t_{2}\vert T_{1}=s_1,T_1<T_2)f_{T_1}(s_1) ds_1+ \int_{t_0}^{t_2}\mathbb{P}(T_{1}<t_{1},T_{2}<t_{2}\vert T_{2}=s_2,T_2<T_1)f_{T_2}(s_2) ds_2\nonumber}\\
&& \qquad=\nonumber \int_{t_0}^{t_{1}}\int_{-\infty }^{B_{2}}\mathbb{P}(T_{2}< t_{2}\vert T_{1}=s_1,X_{2}^{a}(s_1)=x_2)f^a_{X_2|T_1}(x_2,s_1) f_{T_1}(s_1) dx_2 ds_1\\
&&\qquad+\nonumber\int_{t_0}^{t_{2}}\int_{-\infty }^{B_{1}}\mathbb{P}(T_{1}< t_{1}\vert T_{2}=s_2,X_{1}^{a}(s_2)=x_1)f^a_{X_1|T_2}(x_1,s_2) f_{T_2}(s_2) dx_1 ds_2\\
 &&\qquad = \int_{t_0}^{t_1}\int_{-\infty }^{B_{2}}\mathbb{P}(T_{2}<t_{2}|(B_1,x_2),s_1)f^a_{(X_2,T_1)}(x_2,s_1) dx_2 ds_1+
\int_{t_0}^{t_2}\int_{-\infty }^{B_{1}}\mathbb{P}(T_{1}<t_{1}|(x_1,B_2),s_2)f^a_{(X_1,T_2)}(x_1,s_2) dx_1 ds_2,\nonumber
\\[-1.5ex]
\label{con1}
\end{eqnarray}
where in the second equality we condition on the value of the component which has not yet reached its level at the time when the other component crosses its boundary, and in the last equality we use the Markov property, which holds because $\mathbf{X}$ and thus $\mathbf{X}^{a}$ are Markov processes. Finally, \eqref{ProbTT} follows by differentiating \eqref{con1} with respect to $t_1$ and $t_2$. 

In presence of crossing boundary $\bm B$, we assume that $X_1$ crosses $B_1$ at time $T_1=s_1<T_2$ and $\bm X(s_1)=(B_1,x_2)$. Then both components evolve and $X_2$ crosses $B_2$ at time $T_2$. Therefore  
\begin{equation}\label{conto2}
\mathbb{P}(T_{2}<t_{2}\vert \bm (B_1,x_2), s_1)=\int_{s_1}^{t_2} f_{T_2}\left(s_2|\bm (B_1,x_2),s_1\right) ds_2=\int_{s_1}^{t_2}\int_{-\infty}^{\infty} f_{(X_1,T_2)}\left((x_1,B_2),s_2|(B_1,x_2),s_1\right) dx_1  ds_2.
\end{equation}
A similar expression holds when $T_2=s_2<T_1$. Since \eqref{con1} is still valid  when $\bm B$ is absorbing, \eqref{ProbTT3} follows by plugging \eqref{conto2} and $\mathbb{P}(T_1<t_1|(x_1,B_2),s_2)$ into \eqref{con1}, and differentiating \eqref{con1} with respect to $t_1$ and $t_2$. 

\section*{Appendix B: Proof of Theorem \ref{Volt}}\label{appC}
Consider the exit times of the process $\mathbf{X}$. The survival distribution of $\mathbf{X}$, for $x_1>B_1$ and $x_2>B_2$, is given by
\begin{eqnarray}
\nonumber&&\bar F_\mathbf{X}(\mathbf{x},t)= \mathbb{P}\left(\mathbf{X}(t)\geq \mathbf{x},T_{1}<T_{2}\right)+ \mathbb{P}\left(\mathbf{X}(t)\geq \mathbf{x},T_{1}>T_{2}\right)\\
\nonumber&&\qquad=\int^{t}_{t_0} \int_{-\infty}^{B_2}\mathbb{P}\left(\mathbf{X}(t)\geq \mathbf{x}\vert T_1=\tau, {X_2(T_1)=y}\right)f^a_{(X_{2}, T_{1})}\left( y,\tau\right) dy d\tau \nonumber\\
 &&\qquad+\int^{t}_{t_0} \int_{-\infty}^{B_1}\mathbb{P}\left(\mathbf{X}(t)\geq \mathbf{x}, \vert T_2=\tau, {X_1(T_2)=y}\right)f^a_{(X_{1}, T_{2})}\left( y,\tau\right) dy d\tau\nonumber\\
&&\nonumber\qquad= \int^{t}_{t_0} \int_{-\infty}^{B_2}\mathbb{P}\left(\mathbf{X}(t)\geq \mathbf{x}\vert\ (B_1,y)\right)f^a_{(X_{2}, T_{1})}\left( y,\tau\right)
 dy d\tau+\int^{t}_{t_0} \int_{-\infty}^{B_1}\mathbb{P}\left(\mathbf{X}(t)\geq \mathbf{x}\vert(y,B_2)\right) f^a_{(X_{1}, T_{2})}\left( y,\tau\right) dy d\tau,\\[-1.5ex]\label{Fdim}
\end{eqnarray} 
where the last equality follows by the strong Markov property. Then \eqref{Volterra} follows by choosing 
$\mathbf{x}=(x_1,B_2)$ and $\mathbf{x}=(B_1,x_2)$, respectively, and \eqref{Volterradensity} follows by differentiating (\ref{Fdim}) with respect to $\mathbf{x}$.

\section*{Appendix C: Proof of Lemma \ref{lemma1}}\label{appD}
Let us focus on $E^{(1)}_{k,u_1}$ in \eqref{errorEa}. Subtracting \eqref{ftilda1} from \eqref{fcaplimit}, we obtain
\begin{subequations}\label{2}
\begin{align}
&\hat{f}^a_{(X_{1}, T_{2}) }\left(  x_1 , t_k\right)-\tilde{f}^a_{(X_{1}, T_{2}) }\left(  x_1 , t_k\right)=\sum_{\rho=0}^{k-1}\left[-\int_{-\infty}^{B_1} \frac{\partial\bar{F}_{\mathbf{X}}((x_1,B_2),t_k\vert( y,B_2),t_\rho)}{\partial x_1}\hat{f}^a_{(X_{1}, T_{2}) }\left( y , t_\rho\right)dy \right. \nonumber\\
&-  \int_{-\infty}^{B_2} \frac{\partial\bar{F}_{\mathbf{X}}((x_1,B_2),t_k\vert(B_1,y),t_\rho)}{\partial x_1}\hat{f}^a_{(X_{2}, T_{1}) }\left( y , t_\rho\right)dy +r_1 \sum_{u_1=1}^{m_1} \frac{\partial\bar{F}_{\mathbf{X}}((x_1,B_2),t_k\vert( y_{u_1},B_2),t_\rho)}{\partial x_1}\tilde{f}^a_{(X_{1}, T_{2}) }\left(y_{u_1}, t_\rho\right) \nonumber\\
&  \left. 
+ r_2\sum_{u_2=1}^{m_2} \frac{\partial\bar{F}_{\mathbf{X}}((x_1,B_2),t_k\vert(B_1,y_{u_2}),t_\rho)}{\partial x_1}\tilde{f}^a_{(X_{2}, T_{1}) }\left(y_{u_2} , t_\rho\right)\right]. \label{2a}
 \end{align}
\end{subequations}
Note that the term on the left hand side of \eqref{2a} is equal to the term on the right hand side for $\rho=k$, due to conditions \eqref{Flimit}. Rewriting  it with respect to $t_{k-1}$, we obtain
\begin{subequations}\label{3}
\begin{align}
&\sum_{\rho=0}^{k-1}\left[-\int_{-\infty}^{B_1} \frac{\partial\bar{F}_{\mathbf{X}}((x_1,B_2),t_{k-1}\vert( y,B_2),t_\rho)}{\partial x_1}\hat{f}^a_{(X_{1}, T_{2}) }\left(  y , t_\rho\right)dy \right. -  \int_{-\infty}^{B_2} \frac{\partial\bar{F}_{\mathbf{X}}((x_1,B_2),t_{k-1}\vert(B_1,y),t_\rho)}{\partial x_1}\hat{f}^a_{(X_{2}, T_{1}) }\left( y , t_\rho\right)dy \nonumber\\
&+r_1 \sum_{u_1=1}^{m_1} \frac{\partial\bar{F}_{\mathbf{X}}((x_1,B_2),t_{k-1}\vert( y_{u_1},B_2),t_\rho)}{\partial x_1}\tilde{f}^a_{(X_{1}, T_{2}) }\left(y_{u_1}, t_\rho\right)  \left. + r_2\sum_{u_2=1}^{m_2} \frac{\partial\bar{F}_{\mathbf{X}}((x_1,B_2),t_{k-1}\vert(B_1,y_{u_2}),t_\rho)}{\partial x_1}\tilde{f}^a_{(X_{2}, T_{1}) }\left(y_{u_2} , t_\rho\right)\right]=0. \label{3a}
\end{align}
\end{subequations}
Then Lemma \ref{lemma1} follows by subtracting \eqref{3a} from \eqref{2a} and setting $x_i=y_{u_i}$, for $i=1,2$. The error $E^{(2)}_{k,u_1}$ in \eqref{errorEb} is obtained in analogous way.\qed

\section*{Appendix D: Proof of Theorem \ref{error}}\label{appE}
At first we study the error $E^{(i)}_{k,u_i}$ due to the spatial discretization. It can be decomposed as
\begin{equation}\label{E}
E^{(i)}_{k,u_i}=A ^{(i)}_{k,u_i}- B^{(i)}_{k,u_i}, \qquad k=1,\ldots, N.
\end{equation}
Here, $A^{(i)}_{k,u_i}$ has the same expression as $E^{(i)}_{k,u_i}$ in (\ref{errorE}), replacing $\tilde{f}^a_{(X_{i},T_{j})}\left( y, t_j\right)$ with $\hat{f}^a_{(X_{i},T_{j})}\left( y, t_j\right)$. Moreover, $B_{k,u_i}^{(i)}$ is defined by
\begin{subequations}\label{B}
\begin{align}
 B^{(1)}_{k,u_1}&=\sum_{\rho=0}^{k-1}\left[ r_1 \sum_{v_1=1}^{m_1} K_{1,k,\rho}((y_{u_1},B_2),(y_{v_1},B_2))E^{(1)}_{\rho,v_1}+r_2 \sum_{v_2=1}^{m_2} K_{1,k,\rho}((y_{u_1},B_2),(B_1,y_{v_2}))E^{(2)}_{\rho,v_2}
 \right],\label{Ba}\\
B^{(2)}_{k,u_2}&=\sum_{\rho=0}^{k-1}\left[r_1 \sum_{v_1=1}^{m_1} K_{2,k,\rho}((B_1, y_{u_2}),(y_{v_1},B_2))E^{(1)}_{\rho,v_1}+r_2 \sum_{v_2=1}^{m_2} K_{2,k,\rho}((B_1, y_{u_2}),(B_1,y_{v_2}))E^{(2)}_{\rho,v_2}
 \right].
\end{align}
\end{subequations}
The term $A ^{(i)}_{k,u_i}$ accounts for the approximation of the spatial integrals with finite sums. Hence we can split it into two terms: the first, denoted by $A ^{(i,a)}_{k,u_i}$, accounts for the discretization procedure; the second, denoted by $A ^{(i,b)}_{k,u_i}$, accounts for the truncation of the series. Let us focus on $A^{(1,a)}_{k,u_1}$. By definition, we have
\begin{subequations}\label{Aa}
\begin{align}
 |A^{(1,a)}_{k,u_1}|&=\left|\sum_{\rho=0}^{k-1}\left\{\left[ \int_{-\infty}^{B_1}K_{1,k,\rho}((y_{u_1},B_2),(y,B_2)) 
 \hat{f}^a_{(X_{1}, T_{2}) }( y , t_\rho) dy\right. \nonumber\right. -r_1 \sum_{v_1=1}^{\infty} K_{1,k,\rho}((y_{u_1},B_2),(y_{v_1},B_2))\hat{f}^a_{(X_{1}, T_{2}) }( y_{v_1} , t_\rho)\right]\nonumber\\
 &+\left[ \int_{-\infty}^{B_2} K_{1,k,\rho}((y_{u_1},B_2),(B_1,y))
\hat{f}^a_{(X_{2},T_{1})}\left( y, t_\rho\right)dy \left.\left.-r_2 \sum_{v_2=1}^{\infty} K_{1,k,\rho}((y_{u_1},B_2),(B_1,y_{v_2}))\hat{f}^a_{(X_{2},T_{1})}\left( y_{v_2}, t_\rho\right) 
 \right]\right\} \right|.\label{Aaa}
  \end{align}
\end{subequations}
Considering the terms in the first square brackets in \eqref{Aaa}, we have
\begin{align}
&\left| \int_{-\infty}^{B_1}K_{1,k,\rho}((y_{u_1},B_2),(y,B_2)) 
 \hat{f}^a_{(X_{1}, T_{2}) }( y , t_\rho) dy -r_1 \sum_{v_1=1}^{\infty} K_{1,k,\rho}((y_{u_1},B_2),(y_{v_1},B_2))\hat{f}^a_{(X_{1}, T_{2}) }( y_{v_1} , t_\rho)\right|\nonumber\\
&\leq \left| \int_{B_1-r_1}^{B_1}K_{1,k,\rho}((y_{u_1},B_2),(y,B_2)) 
 \hat{f}^a_{(X_{1}, T_{2}) }( y , t_\rho) dy\right| \leq  h \int_{B_1-r_1}^{B_1}C_{1,1}(y) \left|\hat{f}^a_{(X_{1}, T_{2}) }( y , t_\rho)\right|dy \leq  h r_1 \eta_{1,1},\label{ineqfin}
 \end{align}
where we used assumption (i) and eq. (3.4.5) in \cite{Davis} in the first inequality and assumption (ii) in the second. Then, thanks to regularity condition (ii), it follows that $\hat f^a_{(X_j,T_i)}(y,t_\rho)$ is bounded. Moreover, the integrable function $C_{1,1}(y)$ on the compact  interval $[B_1-r_1,B_1]$ is bounded. Thus $C_{1,1}(y) \vert \hat{f}^a_{(X_{1}, T_{2}) }( y , t_\rho)\vert\leq \eta_{1,1}$ for a positive constant $\eta_{1,1}$, yielding \eqref{ineqfin}.  A similar procedure can be done for the terms in the second square brackets in \eqref{Aaa} and for $|A^{(2,a)}_{k,u_2}|$, obtaining:
\begin{equation}\label{Aa1}
 |A^{(i,a)}_{k,u_i}|\leq(r_1\eta_{i,1} +r_2\eta_{i,2} )\sum_{\rho=0}^{k-1}h =(r_1\eta_{i,1} +r_2\eta_{i,2} )  t_{k}, \qquad i=1,2,
\end{equation}
where $\eta_{l,i}$ are suitable positive constants given by $C_{l,i}|f^a_{(X_i,T_j)}(y,t_\rho)|\leq \eta_{l,i}$, for $i,l,j=1,2, i \neq j$ and $t_k=h k$. \newline
Consider the error $A ^{(i,b)}_{k,u_i}$. Using assumption (i), eq. (3.4.5) in \cite{Davis} and then assumptions (ii), (iii) in sequence, we get
\begin{subequations}\label{Ab}
\begin{align}
 |A^{(1,b)}_{k,u_1}|&=\left|\sum_{\rho=0}^{k-1}\left[r_1 \sum_{v_1=m_1+1}^{\infty} K_{1,k,\rho}((y_{u_1},B_2),(y_{v_1},B_2))\hat{f}^a_{(X_{1}, T_{2}) }( y_{v_1} , t_\rho) \right.\right. \left. \left.+r_2 \sum_{v_2=m_2+1}^{\infty} K_{1,k,\rho}((y_{u_1},B_2),(B_1,y_{v_2}))\hat{f}^a_{(X_{2},T_{1})}\left( y_{v_2}, t_\rho\right)  \right]\right|
\nonumber\\
&\leq \left|\sum_{\rho=0}^{k-1}h\left[\int_{-\infty}^{B_1-m_1 r_1} C_{1,1}(y)\hat{f}^a_{(X_{1}, T_{2}) }( y , t_\rho)dy + \int_{-\infty}^{B_2-m_2 r_2} C_{1,2}(y)\hat{f}^a_{(X_{2},T_{1})}\left( y, t_\rho\right)dy  \right]\right| \leq (\psi_{1,1} r_1+\psi_{1,2}r_2)t_{k},\label{Aba}\\
|A^{(2,b)}_{k,u_2}| &\leq (\psi_{2,1} r_1+\psi_{2,2} r_2)t_{k}.\label{Abb}
\end{align}
\end{subequations}
From (\ref{Aa1}), (\ref{Ab}) and $r=\max(r_1,r_2)$, we get $|A^{(i)}_{k,u_i}|\leq r G_i t_{k}$, where $G_i, i=1,2$ are suitable positive constants.
Using this bound in (\ref{E}) and observing that $B_{k,u_i}^{(i)}$ in (\ref{B}) involves the errors $E^{(i)}_{\rho,v_i}$ for $0\leq \rho \leq k-1$, we get a system of inequalities
\begin{subequations}\label{E1}
\begin{align}
|E^{(1)}_{k,u_1}|&\leq G_1 r t_{k}+r\sum_{\rho=0}^{k-1}\left[ \sum_{v_1=1}^{m_1} |K_{1,k,\rho}((y_{u_1},B_2),(y_{v_1},B_2))||E^{(1)}_{\rho,v_1}| + \sum_{v_2=1}^{m_2} |K_{1,k,\rho}((y_{u_1},B_2),(B_1,y_{v_2}))||E^{(2)}_{\rho,v_2}| \right],\\
 |E^{(2)}_{k,u_2}|&\leq G_2 r t_{k}+r\sum_{\rho=0}^{k-1}\left[ \sum_{v_1=1}^{m_1} |K_{2,k,\rho}((B_1, y_{u_2}),(y_{v_1},B_2))||E^{(1)}_{\rho,v_1}|+ \sum_{v_2=1}^{m_2} |K_{2,k,\rho}((B_1, y_{u_2}),(B_1,y_{v_2}))||E^{(2)}_{\rho,v_2}|
\right].
\end{align}
\end{subequations}
We extend the method proposed in \cite{CMV} to the system \eqref{E1}, that we iteratively solve:
\begin{align}
|E_{0,u_i}^{(i)}|&=0:=r p_0^{(i)};\nonumber\\
|E_{1,u_i}^{(i)}|&\leq G_1 r t_1=:r p_1^{(i)};\nonumber\\
|E_{2,u_1}^{(1)}|&\leq G_1 r t_2+r\left[ r p_1^{(1)} \sum_{v_1=1}^{m_1} |K_{1,k,\rho}((y_{u_1},B_2),(y_{v_1},B_2))| +r p_1^{(2)} \sum_{v_2=1}^{m_2} |K_{1,k,\rho}((y_{u_1},B_2),(B_1,y_{v_2}))|
 \right]\nonumber\\
 &\leq r \left[ G_1 t_2+r \beta_{1,1} p_1^{(1)} +r\beta_{1,2} p_1^{(2)} \right]=:r p_2^{(1)},\label{disEi}\\
|E_{2,u_2}^{(2)}|&\leq r\left[ G_2 t_2+ r \beta_{2,1} p_1^{(1)} +r\beta_{2,2} p_1^{(2)} \right]:=r p_2^{(2)},\nonumber
\end{align}
where \eqref{disEi} holds due to assumption (ii) and eq. (3.4.5) in \cite{Davis}. Here $\beta_{i,l}, i,l=1,2$ are suitable  positive  constants, which depend neither on $r$ nor on $h$. Iterating this procedure, (\ref{E1}) becomes
\begin{equation}\label{Ep}
|E^{(i)}_{k,u_i}|\leq r\left[G_i t_{k}+ r\sum_{\rho=0}^{k-1} \left( \beta_{i,1} p_{\rho}^{(1)}+\beta_{i,2} p_{\rho}^{(2)}\right)  \right]:=r p_{k}^{(i)}, \qquad i=1,2.\end{equation}
Since $t_k\leq \Theta$, from \eqref{Ep} it follows
$$
p_k^{(i)}\leq G_i \Theta +r\sum_{\rho=0}^{k-1} \left(\beta_{i,1} p_\rho^{(1)}+\beta_{i,2}p_\rho^{(2)}\right), \qquad i=1,2.
$$
Then, by eq. (7.18) in \cite{Li}, we get $p_k^{(i)}\leq G_i\Theta \exp[(\beta_{i,1}+\beta_{i,2})rt_k] $. Therefore
\begin{equation}\label{E_G}
|E^{(i)}_{k,u_i}|\leq r G_i \Theta\exp\left[(\beta_{i,1}+\beta_{i,2})rt_k\right], \qquad i=1,2
\end{equation}
implying $\vert E^{(i)}_{k,u_i}\vert=\textrm{O}(r)$. 

\noindent Consider now the time discretization error $e^{(i)}_k(y_{u_i})$, focusing on $e^{(1)}_k(y_{u_1})$. The error formulas for the Euler method are
\begin{eqnarray}\label{delta}
\delta_{1,1,k}(h)&=\frac{ht_k}{2} \int_{-\infty}^{B_1} \left. \frac{\partial}{\partial t}\bar{F}_{\mathbf{X}}((y_{u_1},B_2),t_k\vert(y,B_2),t)f^a_{(X_{1}, T_{2})}\left( y, t\right)dy \right|_{t=\tau}; \\
\delta_{1,2,k}(h)&=\frac{ht_k}{2} \int_{-\infty}^{B_2} \left.  \frac{\partial}{\partial t} \bar{F}_{\mathbf{X}}(( y_{u_1},B_2),t_k\vert(B_1,y),t)f^a_{(X_{2}, T_{1})}\left( y, t\right)dy \right|_{t=\tau};\nonumber\\
\delta_{2,1,k}(h)&=\frac{ht_k}{2} \int_{-\infty}^{B_1} \left. \frac{\partial}{\partial t} \bar{F}_{\mathbf{X}}((B_1, y_{u_2}),t_k\vert(y,B_2),t)f^a_{(X_{1}, T_{2})}\left( y, t\right)dy\right|_{t=\tau};\nonumber\\ 
\delta_{2,2,k}(h)&=\frac{ht_k}{2}\int_{-\infty}^{B_2} \left.  \frac{\partial}{\partial t}\bar{F}_{\mathbf{X}}((B_1, y_{u_2}),t_k\vert(B_1,y),t)f^a_{(X_{2}, T_{1})}\left( y, t\right)dy\right|_{t=\tau}\nonumber,
\end{eqnarray}
where $\tau \in (0,\Theta)$ and $t_k=hk$. Rewriting (\ref{Volterra1discret}) with the corresponding residuals, and evaluating it in $x_1=y_{u_1}$, we get 
\begin{subequations}
\begin{align}
 \bar{F}_{\mathbf{X}}((y_{u_1},B_2),t_k) &=h\sum_{\rho=0}^k \int_{-\infty}^{B_1} \bar{F}_{\mathbf{X}}((y_{u_1},B_2),t_k\vert( y,B_2),t_\rho){f}^a_{(X_{1}, T_{2}) }\left(  y , t_\rho\right)dy\nonumber\\
&+ h\sum_{\rho=0}^k \int_{-\infty}^{B_2} \bar{F}_{\mathbf{X}}((y_{u_1},B_2),t_k\vert(B_1,y),t_\rho){f}^a_{(X_{2}, T_{1}) }\left( y , t_\rho\right) dy+\delta_{1,1,k}(h)+\delta_{1,2,k}(h).
\label{Volterra1discret resid}
  \end{align}
\end{subequations}
Subtracting (\ref{Volterra1discret}) from (\ref{Volterra1discret resid}) and differentiating with respect to $y_{u_1}$, we get the integral equation for $e^{(1)}_\rho(y)$ 
\begin{subequations}\label{e}
\begin{align}
 -\frac{\partial}{\partial y_{u_1}}\left[\delta_{1,1,k}(h)+\delta_{1,2,k}(h)\right]=h\sum_{\rho=0}^{k} \left[\frac{\partial}{\partial y_{u_1}}\int_{-\infty}^{B_1}\bar{F}_{\mathbf{X}}((y_{u_1},B_2),t_k\vert(y,B_2),t_\rho) e^{(1)}_\rho(y)dy + \frac{\partial}{\partial y_{u_1}}\int_{-\infty}^{B_2}\bar{F}_{\mathbf{X}}((y_{u_1},B_2),t_k\vert(B_1, y),t_\rho) e^{(2)}_\rho(y)dy\right]. \label{ea}
\end{align}
\end{subequations}
Rewriting \eqref{ea} with respect to $k-1$, subtracting it from \eqref{ea} and using \eqref{Flimit}, we obtain
\begin{subequations}\label{e1}
\begin{align}
e^{(1)}_k(y_{u_1})&-\sum_{\rho=0}^{k-1}\left[ \int_{-\infty}^{B_1}K_{1,k,\rho}((y_{u_1},B_2),(y,B_2)) e^{(1)}_\rho(y)dy+\int_{-\infty}^{B_2}K_{1,k,\rho}((y_{u_1},B_2),(B_1,y)) e^{(2)}_\rho(y)dy\right]\nonumber\\
&=\frac{\partial}{\partial y_{u_1}}\left[\frac{(\delta_{1,1,k}(h)-\delta_{1,1,k-1}(h))+(\delta_{1,2,k}(h)-\delta_{1,2,k-1}(h))}{h}\right].\label{e1a}
\end{align}
\end{subequations}
Using \eqref{kernel}, \eqref{delta} and assumption (v), and since $t_{k-1}=t_k-h$, we have
\begin{eqnarray*}
\frac{\partial}{\partial y_{u_1}}\left|\delta_{1,1,k}(h)-\delta_{1,1,k-1}(h)\right|&\leq&\frac{ht_{k}}{2} \left. \int_{-\infty}^{B_1} \frac{\partial}{\partial t} \left[\vert K_{1,k,t}((y_{u_1},B_2),(y,B_2))\vert 
\vert f^a_{(X_{1}, T_{2})}\left( y, t\right)\vert  dy \right]\right|_{t=\tau}\\
&+&\frac{h^2}{2}\left|\frac{\partial}{\partial y_{u_1}}\left.\int_{-\infty}^{B_1} \frac{\partial}{\partial t}\left[\bar{F}_{\mathbf{X}}((y_{u_1},B_2),t_{k-1}\vert (y,B_2),t)f^a_{(X_{1}, T_{2})}\left( y, t\right)dy\right]\right|_{t=\tau}\right|\\
&\leq& \frac{h^2 }{2} \left[ t_{k} Q_{1,1} +S_{1,1}\right]:=\frac{h^2}{2} \alpha_{1,1}.
\end{eqnarray*}
The last inequality holds applying assumptions (ii) and (iv) on the first term, and assumption (v) on the second term, for a  suitable positive constant $S_{1,1} $.
Then \eqref{e1} becomes
\begin{subequations}\label{e2}
\begin{align}
|e^{(1)}_k(y_{u_1})|&\leq\frac{(\alpha_{1,1}+\alpha_{1,2})ht_{k}}{2}+\sum_{\rho=0}^{k-1}\left[\int_{-\infty}^{B_1}\left|K_{1,k,\rho}((y_{u_1},B_2),(y,B_2)) e^{(1)}_\rho(y)\right|dy+ \int_{-\infty}^{B_2}\left|K_{1,k,\rho}((y_{u_1},B_2),(B_1,y)) e^{(2)}_\rho(y)\right|dy\right],\label{e2a}\\
|e^{(2)}_k(y_{u_2})|&\leq\frac{(\alpha_{2,1}+\alpha_{2,2})ht_{k}}{2}+\sum_{\rho=0}^{k-1}\left[ \int_{-\infty}^{B_1}\left|K_{2,k,\rho}((B_1, y_{u_2}),(y,B_2)) e^{(1)}_\rho(y)\right|dy+\int_{-\infty}^{B_2}\left|K_{2,k,\rho}((B_1, y_{u_2}),(B_1,y)) e^{(2)}_\rho(y)\right|dy\right], \label{e2b}
\end{align}
\end{subequations}
where \eqref{e2b} is obtained as \eqref{e2a}. Setting $\gamma_l=\max\{\alpha_{l,1},\alpha_{l,2}\}, l=1,2$, we can  iteratively solve \eqref{e2}  for $k\geq 0$:
\begin{eqnarray*}
\vert e_0^{(i)}\vert &=&0:=h q_0^{(i)};\\
\vert e_1^{(i)}\vert &\leq & \gamma_i h t_1:= h q_1^{(i)};\\
\vert e_2^{(1)}\vert &\leq & \gamma_1 h t_2 +h q_1^{(1)}\int_{-\infty}^{B_1} \vert K_{1,k,\rho}((y_{u_1},B_2),(y,B_2))\vert dy+ h q_1^{(2)}\int_{-\infty}^{B_2} \vert K_{1,k,\rho}((y_{u_1},B_2),(B_1,y))\vert dy\\
&\leq& h\left (\gamma_1 t_2  + h \xi_{1,1} q_1^{(1)}+ h  \xi_{1,2} q_1^{(2)} \right) := h q_2^{(1)},\\
\vert e_2^{(2)}\vert &\leq &  h\left (\gamma_2 t_2  + h \xi_{2,1} q_1^{(1)} + h \xi_{2,2} q_1^{(2)}  \right) := h q_2^{(2)},
\end{eqnarray*}
where we used assumption (ii) to bound $e_2^{(i)}$. Here $\xi_{i,l}, i,l=1,2 $ are suitable positive constants independent on $h$ and $r$. In general
\begin{equation}\label{e3}
|e^{(i)}_k(y_{u_i})|\leq h \left[\gamma_i t_{k}+h\left( \xi_{i,1}\sum_{\rho=0}^{k-1}q_\rho^{(1)}+ \xi_{i,2}\sum_{\rho=0}^{k-1}q_\rho^{(2)}\right)\right]:=hq_k^{(i)}, \qquad i=1,2. 
\end{equation}
Since $t_k \leq \Theta$, from \eqref{e3} it follows
\[
q_k^{(i)}\leq \gamma_i \Theta + h\left( \xi_{i,1} \sum_{\rho=0}^{k-1} q_\rho^{(1)}+ \xi_{i,2} \sum_{\rho=0}^{k-1} q_\rho^{(2)}\right), \qquad i=1,2,
\]
and applying eq. (7.18) in \cite{Li}, we get 
$q_k^{(i)}\leq \gamma_i \Theta \exp[(\xi_{i,1}+\xi_{i,2})t_k], i=1,2$ and thus 
\begin{subequations}\label{e_G}
\begin{align}
|e^{(i)}_k(y_{u_i})|&\leq h \gamma_i \Theta\exp\left[(\xi_{i,1}+\xi_{i,2})t_k\right], \qquad i=1,2.
\end{align}
\end{subequations}
The theorem follows by noting that $|e^{(i)}_k(y_{u_i})|=O(h)$. 

\section*{Appendix E: Proof of Lemma \ref{corfa}}\label{appG}
Consider $j=1, i=2$. When $x_{1}\rightarrow B_{1}$, both $f^a_{X_1}$  and $f^a_{\mathbf{X}}$ go to zero, due to the boundary condition \eqref{abs}. Therefore $f^a_{X_2\vert X_1}$ is indefinite. From the definition of $\phi$, we have that $\phi \to \alpha$ when $x_1\to B_1$, and thus $\sin(n\pi\phi/\alpha)\to 0$ and $H(\bar{r},\bar{r}_0,\phi,\phi_0,t)\to 0$. Moreover, 
\[
\left[ 1-\exp \left( \frac{2(B_{1}-x_{01})(x_{1}-B_{1})}{\sigma
_{1}^{2}t}\right) \right] \to 0, 
\]
 when $x_1\to B_1$. Hence, the last two terms in \eqref{facondXWiener} produce an indefinite form. Applying l'H\'o{}pital's rule, we obtain
\[
\lim_{x_1\rightarrow B_{1}}\frac{\sin \left( \frac{n\pi \phi }{\alpha }%
\right) }{1-\exp \left(\frac{2(B_{1}-x_{01})(x_{1}-B_{1})}{\sigma _{1}^{2}t%
}\right) }= 
\frac{\sigma _{1}\sigma _{2}\pi \sqrt{1-\rho ^{2}}t}{2\alpha
(B_{2}-x_2)(B_{1}-x_{01})}n\delta_2.\nonumber
\]
The result follows by plugging this ratio into \eqref{facondXWiener}. The density $f^a_{X_1|T_2}$ is derived in the same way, noting that $\phi\to 0$ when $x_2\to B_2$, and the coefficient $\delta_1=1$ instead of $\delta_2$ is obtained applying  l'H\'o{}pital's rule.

\end{document}